\documentclass[a4paper,reqno]{amsart}
\usepackage{ifthen}
\usepackage{amsmath}
\usepackage{amssymb}    
\usepackage{stmaryrd} 
\usepackage{eufrak}     
\usepackage{mathrsfs} 
\usepackage{bbold}    
\provideboolean{usemathrsfs}
\setboolean{usemathrsfs}{true}
\usepackage{enumerate}
\usepackage{xspace}
\usepackage{boxedminipage}
\usepackage[utf8]{inputenc}
\usepackage{verbatim}
\provideboolean{usebeamer}
\provideboolean{isthesis}
\provideboolean{isamsltex}
\provideboolean{issiamltex}

\usepackage{hyperref}
\newcommand{\linkedurl}[1]{{\color{blue}\href{#1}{#1}}}
\newcommand{\linkedemail}[1]{{\makeatletter\color{blue}\href{mailto:#1}{#1}\makeatother}}
\newcommand{\Ignore}[1]{}
\newcommand{\freeze}[1]{}
\newcommand{\crossout}[1]{{\textcolor{red!20}{#1}}}
\newcommand{\highlight}[1]{{\color{blue}#1}}
\newcommand{\standout}[1]{{\textcolor{magenta}{#1}}}
\provideboolean{shownotes}
\setboolean{shownotes}{true}
\newcounter{margnote}[page]
\newcommand{\margnotemark}{\ensuremath{\text{\standout{\upshape\texttt{>\arabic{margnote}<}}}}}
\newcommand{\margnote}[2][]{
  \ifthenelse{
    \boolean{shownotes}
  }{
    \stepcounter{margnote}
    \margnotemark
    \marginpar{
      \texttt{
        \begin{minipage}{2cm}
          \raggedright\tiny
          \margnotemark{#1}: 
          #2
        \end{minipage}
      }
    }
  }{
  }
}
\newcommand{\mathnote}[2][]{
  \ifthenelse{
    \boolean{shownotes}
  }{
    \stepcounter{margnote}
    \margnotemark
    \text{
      \standout{
        \texttt{
          \tiny
          #2
        }
      }
    }
  }{
  }
}
\provideboolean{showtodo}

\newcommand{\todo}[1]{\ifthenelse{\boolean{showtodo}}{\margnote[To do.]{#1}}{}}
\newcommand{\Todo}[1]{
  \ifthenelse{\boolean{showtodo}}{
    \setlength{\parindent}{0pt}\ 
    \par
    \begin{boxedminipage}{\textwidth}
      \texttt{#1}
    \end{boxedminipage}
    \par
  }{}}

\provideboolean{showcomments}
\newcommand{\margincomment}[1]{
\ifthenelse{\boolean{showcomments}}{\marginpar{\tiny #1}}{}
}
\provideboolean{showchanges}
\setboolean{showchanges}{false}
\newcommand{\changes}[1]{
  \ifthenelse{\boolean{showchanges}} {{\highlight{#1}}} {#1}
}
\newcommand{\changefromto}[3][replace with]{
  \ifthenelse{\boolean{showchanges}}
  {{\crossout{#2}\margnote{#1}}{\highlight{#3}}}
  {#3\xspace}
}
\newcommand{\ChangePar}[2]{
  \ifthenelse{\boolean{showchanges}}
  {{\par$\mapsfrom$ \textcolor{red!20}{#1}}{\par$\mapsto$ \textcolor{blue}{#2}}}
  {\par #2}
}
\newcommand{\InsertPar}[1]{
  \ifthenelse{\boolean{showchanges}}
  {{\par$\mapsto$ \textcolor{blue}{#1}}}
  {\par #1}
}

\newcommand{\aposteriori}{{aposteriori}\xspace}
\newcommand{\Aposteriori}{{Aposteriori}\xspace}

\usepackage{mathrsfs}
\provideboolean{usemathrsfs}
\setboolean{usemathrsfs}{true}
\newcommand{\mathscript}
	   {\mathscr}

 \newcommand{\cA}{\ensuremath{\mathscript A}\xspace}

 \newcommand{\cE}{\ensuremath{\mathscript E}\xspace}

 \newcommand{\cI}{\ensuremath{\mathscript I}\xspace}

 \newcommand{\cS}{\ensuremath{\mathscript S}\xspace}
 \newcommand{\cT}{\ensuremath{\mathscript T}\xspace}

 \newcommand{\rR}{\ensuremath{\mathbb R}\xspace}

 \newcommand{\reals}{\rR}
 \newcommand{\R}[1]{\reals^{#1}}

 \newcommand{\C}[1]{\complex^{#1}}

 \newcommand{\one}{\ensuremath{\mathbb 1}\xspace}
 \newcommand{\charfun}[1]{\one_{#1}}

 \newcommand{\Tolto}{\smallsetminus}
 \newcommand{\take}{\Tolto}

 \newcommand{\inner}{\cdot}

 \newcommand{\Ga}{\ensuremath{\varGamma}\xspace}

 \newcommand{\Pg}{\ensuremath{\varPi}\xspace}

 \newcommand{\W}{\ensuremath{\varOmega}\xspace}

 \newcommand{\epsi}{\ensuremath{\epsilon}\xspace}

 \newcommand{\w}{\ensuremath{\omega}\xspace}

 \newcommand{\qp}[1]{\ensuremath{\!\left({#1}\right)}}
 \newcommand{\qpreg}[1]{\ensuremath{(#1)}}

 \newcommand{\qpbigg}[1]{\ensuremath{\bigg(#1\bigg)}}
 
 \newcommand{\qb}[1]{\ensuremath{\!\left[{#1}\right]}}

 \newcommand{\qc}[1]{\ensuremath{\left\{{#1}\right\}}}

 \newcommand{\qa}[1]{\ensuremath{\left\langle{#1}\right\rangle}}

 \newcommand{\opinter}[2]{\ensuremath{\left(#1,#2\right)}\xspace}
 \newcommand{\clinter}[2]{\ensuremath{\left[#1,#2\right]}\xspace}
 \newcommand{\opclinter}[2]{\ensuremath{\left(#1,#2\right]}\xspace}
 
 \newcommand{\powqp}[2]{\ensuremath{\qp{#2}^{\kern -.2em\lower .3ex\hbox{\scriptsize $#1$}}\kern-.3em}}
 \newcommand{\powqpreg}[2]{\ensuremath{\qpreg{#2}^{\kern -.2em\lower .3ex\hbox{\scriptsize $#1$}}\kern-.3em}}
 
 \newcommand{\pregsqrt}[1]{\powqpreg{1\!/\!2}{#1}}

 \newcommand{\abs}[1]{\ensuremath{\left|#1\right|}}
 \newcommand{\Norm}[1]{\ensuremath{\left\|#1\right\|}}

 \newcommand{\ltwop}[2]{\ensuremath{\qa{#1,#2}}}

 \newcommand{\duality}[2]{\ensuremath{\left\langle #1\,\vert\,#2\right\rangle}}
 \newcommand{\pairing}[2]{\ensuremath{\qp{#1,#2}}}
 \newcommand{\ensemble}[2]{\ensuremath{\left\{ #1:\;#2 \right\}}}

 \newcommand{\setof}[1]{\qc{#1}}

 \newcommand{\seq}[1]{\seqof{#1}}
 
 \newcommand{\seqs}[2]{\seq{#1}_{#2}}

 \newcommand{\sumifromto}[3]{\ensuremath{\sum_{#1=#2}^{#3}}}

 \newcommand{\jump}[1]{\ensuremath{\left\llbracket #1\right\rrbracket}}

 \newcommand{\rangefromto}[3]{\ensuremath{#1\integerbetween{#2}{#3}}}

 \renewcommand{\d}{\ensuremath{\,\operatorname{d}\!}}

\newcommand{\registered}%
    {\ensuremath{{}^{\bigcirc\!\;\!\!\!\!\!\!\!\;\text{\sc r}}}}

\newcommand{\AND}{\ensuremath{\text{ and }}}

\newcommand{\diam}{\operatorname{diam}}

\renewcommand{\div}{\operatorname{div}}

\newcommand{\esssup}{\operatorname{ess\,sup}}         

\newcommand{\inverse}[1]{\ensuremath{{#1}^{-1}}}

\newcommand{\fraclpf}[2]{(#1)/{#2}}

\newcommand{\fracl}[2]{{#1}/{#2}}

\newcommand{\Eye}[1]{
  \begin{bmatrix}
  \ifthenelse{#1>1}{
    \ifthenelse{#1>2}{
      \ifthenelse{#1>3}{
        1&0&\dotso&0
        \\
        0&1&\dotso&0
        \\
        \vdots&\vdots&\ddots&\vdots
        \\
        0&0&\dotso&1
      }{
        1&0&0
        \\
        0&1&0
        \\
        0&0&1
      }
    }{
      1&0
      \\
      0&1
    }
  }{
    1
  }
  \end{bmatrix}
}

\newcommand{\pd}[2]{\ensuremath{\partial_{#1}{#2}}\xspace} 

\newcommand{\dt}{\ensuremath{\d_t}}

\newcommand{\pdt}[1]{\pd t{#1}}                       

\renewcommand{\vec}[1]{\ensuremath{\boldsymbol{#1}}}

\newcommand{\geomat}[1]{\vec{#1}}

\newcommand{\mat}[1]{\geomat{#1}}

\newcommand{\boundary}{\partial}

 \newcommand{\CC}{\ensuremath{\operatorname C}\xspace}
 \newcommand{\HH}{\ensuremath{\operatorname H}\xspace}
 \newcommand{\LL}{\ensuremath{\operatorname L}\xspace}

 \newcommand{\cont}[1]{\ensuremath{\CC^{#1}}}

 \newcommand{\leb}[1]{\ensuremath{\LL_{#1}}}

 \newcommand{\sobh}[1]{\ensuremath{\HH^{#1}}}
 \newcommand{\sobhz}[1]{\sobh{#1}_0}

 \newcommand{\fespace}{\rV}
 
 \newcommand{\fes}[1]{\ensuremath{\fespace^{#1}}}

 \newcommand{\fesh}{\ensuremath{\fespace_h}}

 \newcommand{\EOC}{\ensuremath{\operatorname{EOC}}\xspace}

\newcommand{\Forall}{\:\forall\:}

\usepackage{amscd}

\newcommand{\funk}[3]{\ensuremath{#1:#2\to#3}}

\renewcommand{\restriction}[2]{\left.#1\right|_{#2}}

\newcommand{\aka}[1]{(also known as {#1})\xspace}

\newcommand{\Program}[1]{\textsf{#1}\xspace}

\newcommand{\matlab}{\Program{Matlab\registered}}

\newcommand{\secsymbol}{\S}

\newcommand{\secref}[1]{\secsymbol\ref{#1}}

\providecommand{\ListParameters}{}
\renewcommand{\ListParameters}
{
	 \setlength{\topsep}{0em}
	 \setlength{\leftmargin}{0em}
         \setlength{\itemsep}{0ex}
	 \setlength{\parsep}{.5ex}
	 \setlength{\itemindent}{\labelsep}
	 \addtolength{\itemindent}{\labelwidth}
}

\newcounter{LetterListItem}
\renewcommand{\theLetterListItem}{(\alph{LetterListItem})}
\newenvironment{LetterList}%
{
	\begin{list}%
	{\theLetterListItem\ }%
	{\usecounter{LetterListItem}
	 \ListParameters
	}
}%
{\end{list}}
\newcounter{NumberListItem}
\renewcommand{\theNumberListItem}{\arabic{NumberListItem}}
{
	\begin{list}%
	{\theNumberListItem.\ }%
	{\usecounter{NumberListItem}%
	 \ListParameters
	}
}%
{\end{list}}
\newcounter{QuestionListItem}
\renewcommand{\theQuestionListItem}{\textbf{Question \arabic{QuestionListItem}}}
{
	\begin{list}%
	{\theQuestionListItem.\ }%
	{\usecounter{QuestionListItem}%
	 \ListParameters
	}
}%
{\end{list}}
\newcounter{RomanListItem}
\renewcommand{\theRomanListItem}{(\roman{RomanListItem})}
{
	\begin{list}%
	{\theRomanListItem\ }%
	{\usecounter{RomanListItem}
	 \ListParameters
	}
}%
{\end{list}}
\newcounter{StepsItem}
{
	\begin{list}%
	{Step \theStepsItem.\ }%
	{\usecounter{StepsItem}%
	 \ListParameters
	}
}%
{\end{list}}

\providecommand{\grad}{\nabla}
\renewcommand{\grad}{\nabla}

\usepackage{amsthm} 
\ifthenelse{\boolean{isthesis}}{%
  \setcounter{secnumdepth}{1}%
}{
  \setcounter{secnumdepth}{2} 
}
\providecommand{\ListParameters}{}
\renewcommand{\ListParameters}
{
	 \setlength{\topsep}{0em}
	 \setlength{\leftmargin}{0em}
         \setlength{\itemsep}{0ex}
	 \setlength{\parsep}{.5ex}
	 \setlength{\itemindent}{\labelsep}
	 \addtolength{\itemindent}{\labelwidth}
}
\newtheoremstyle{plain}
  {}
  {}
  {\mdseries\slshape}
  {\parindent}
  {\bfseries}
  {.}
  {.5em}
  {}

\newtheoremstyle{note}
  {}
  {}
  {}
  {\parindent}
  {\bfseries}
  {.}
  {.5em}
  {}

\newtheoremstyle{claim}
  {}
  {}
  {\mdseries\slshape}
  {}
  {\bfseries}
  {}
  {.5em}
  {}

\newtheoremstyle{exercise}
  {}
  {}
  {}
  {}
  {\bfseries}
  {.}
  {1em}
  {}

\newtheoremstyle{break}
  {}
  {}
  {}
  {}
  {\bfseries}
  {.}
  {\newline}
  {}

\ifthenelse{\boolean{isthesis}}{%
  
}{%
  
}

\newcommand{\pdfformat}[1]{
  \provideboolean{pdfoutput}
  \setboolean{pdfoutput}{#1}
  \ifthenelse{\boolean{pdfoutput}}%
	     {\typeout{using pdf}
               \usepackage{pdfsync}
	       \usepackage[pdftex]{graphicx,xcolor}
               \providecommand{\graphext}{pdf}
	       \renewcommand{\graphext}{pdf}
	       \providecommand{\graphextex}{pdf_t}
	       \renewcommand{\graphextex}{pdf_t}
	       \usepackage{epsfig}
	       \usepackage{tikz}
	       \usepackage{rotating}
	     }
	     {
	       \typeout{using eps}
	       \usepackage[dvips]{graphicx,xcolor}
               \providecommand{\graphext}{eps}
	       \renewcommand{\graphext}{eps}
	       \providecommand{\graphextex}{eps_t}
	       \renewcommand{\graphextex}{eps_t}
	       \usepackage{epsfig}
	       \usepackage{tikz}
	       \usepackage{rotating}
	     }
}

\usepackage[pdftex]{graphicx,xcolor}
\usepackage{epsfig}
\usepackage{tikz}
\usepackage{rotating}
\usepackage{subfigure}
\setboolean{usemathrsfs}{true}
\setboolean{shownotes}{true}
\setboolean{showtodo}{true}
\setboolean{showchanges}{false}
\renewcommand{\Aposteriori}{{A posteriori}\xspace}
\renewcommand{\aposteriori}{{a posteriori}\xspace}

\newcommand{\indtime}{\ensuremath{\theta}}

\newcommand{\honezw}{{\sobhz1(\W)}} 
\newcommand{\hmonew}{{\sobh{-1}(\W)}} 

\newcommand{\tm}{{t_m}}

\newcommand{\tn}{{t_n}}
\newcommand{\tno}{{t_{n-1}}}
\newcommand{\tnp}[1]{\ensuremath{t_{n+#1}}}

\newcommand{\taun}{\ensuremath{\tau_n}}

\newcommand{\In}{\ensuremath{(\tno,\tn]}}

\newcommand{\elln}{\ensuremath{l_n}}
\newcommand{\ellno}{\ensuremath{l_{n-1}}}

\newcommand{\fn}{\ensuremath{f^n}}

\newcommand{\uzero}{u_0}
\newcommand{\U}{\ensuremath{U}}

\newcommand{\un}{\U^n}
\newcommand{\uno}{\U^{n-1}}

\newcommand{\UZERO}{\U^0}
\newcommand{\interpol}[1]{I^{#1}}
\newcommand{\wn}{\ensuremath{\w^n}\xspace}
\newcommand{\wno}{\ensuremath{\w^{n-1}}\xspace}

\newcommand{\discelop}{A}
\newcommand{\disca}[1]{\discelop^{#1}}

\newcommand{\lproj}[1]{\lprojop^{#1}}

\newcommand{\bbil}[2]{B\qp{#1,#2}}

\providecommand{\enorm}[1]{\ensuremath{\Norm{#1}_a}}

\newcommand{\be}{\begin{equation}}
\newcommand{\ee}{\end{equation}}
\newcommand{\bea}{\begin{eqnarray}}
\newcommand{\eea}{\end{eqnarray}}
\newcommand{\beaa}{\begin{eqnarray*}}
\newcommand{\eeaa}{\end{eqnarray*}}

\def \w{{\bf w}}

\renewcommand{\seq}[1]{\ensuremath{\left\{ #1 \right\}}}

\newcommand{\el}{ \kappa \in \mathscript{T}}

\renewcommand{\diam}{\operatorname{diam}}

\newcommand{\dint}{\text{\rm int}}
\newcommand{\mean}[1]{\{ \kern -1.6mm \{#1\} \kern -1.6mm \}}
\newcommand{\ha}{\frac{1}{2}}

\newcommand{\ud}{\mathrm{d}}
\newcommand{\ndg}[1]{| \kern -.25mm \|{#1}| \kern -.25mm \|}
\newcommand{\ltwo}[2]{\|{#1}\|_{#2}}

\newcommand{\amin}{\alpha_{\flat}}
\newcommand{\amax}{\alpha_{\sharp}}
\newcommand{\amink}{a_{\flat}}
\newcommand{\amaxk}{a_{\sharp}}
\newcommand{\maxcondloc}{K_{\mat a}^{\mathscript{T}}}
\newcommand{\estfunc}{\mathscript{E}_{\operatorname{IP}}}

\renewcommand{\Forall}{\ensuremath{\quad\text{for all }}}
\renewcommand{\AND}{\ensuremath{\quad\text{and}\quad}}
\renewcommand{\rangefromto}[3]{\ensuremath{#1=#2,\dotsc,#3}\xspace}
\renewcommand{\w}{\ensuremath{w}}
\renewcommand{\cont}[1]{\ensuremath{C\ifthenelse{\equal{#1}{0}}{}{^{#1}}}}
\renewcommand{\div}{\ensuremath{\nabla\cdot}}
\newcommand{\dgrad}{\ensuremath{\nabla_{\!\!\cT}}}
\newcommand{\discan}{\disca n}
\newcommand{\discano}{\disca {n-1}}

\newcommand{\opera}{\cA}
\newcommand{\operan}{\opera^n}
\newcommand{\operano}{\opera^{n-1}}

\renewcommand{\In}{\ensuremath{\clinter\tno\tn}}
\renewcommand{\fespace}{\ensuremath{S}}
\newcommand{\ccrfes}[1]{\ensuremath{\check S}_{#1}}

\newcommand{\hopfes}[1]{\ensuremath{\cS_{#1}}}
\newcommand{\hopfesh}{\ensuremath{\cS}}
\providecommand{\enorm}[1]{\ndg{#1}}
\newcommand{\espace}[1]{\leb2(#1;\hopfesh)}
\newcommand{\conpart}{\ensuremath{_c}}
\newcommand{\conpartplus}{\ensuremath{_{c{\scriptscriptstyle{\boldsymbol +}}}}}
\newcommand{\dispart}{\ensuremath{_d}}
\newcommand{\dispartplus}{\ensuremath{_{d{\scriptscriptstyle{\boldsymbol +}}}}}
\newcommand{\epsic}{\ensuremath{\epsi\conpart}}

\newcommand{\operproj}{\Pg}
\renewcommand{\lproj}[1]{\ensuremath{\operproj^{#1}}}

\providecommand{\tnp}[1]{\ensuremath{t_{n+#1}}}
\newcommand{\bilform}{B}
\newcommand{\bbilp}[3]{\ensuremath{\bilform^{#1}\qp{#2,#3}}}

\newcommand{\bbilpsub}[4]{\ensuremath{\bilform^{#1}_{#2}\qp{#3,#4}}}
\newcommand{\ipdg}{\text{IPDG}\xspace}
\newcommand{\indoperasymbol}{\zeta}
\newcommand{\indoperafor}[1]{\indoperasymbol'_{#1}}
\newcommand{\indoperaback}[1]{\indoperasymbol''_{#1}}
\newcommand{\indoperamesh}[1]{\indoperasymbol^\circ_{#1}}
\newcommand{\indopera}[1]{\indoperasymbol_{#1}}
\renewcommand{\indtime}[1]{\theta_{#1}}
\newcommand{\indcoarse}[1]{\gamma_{#1}}
\newcommand{\indfosc}[1]{\beta_{#1}}
\newcommand{\indftimeosc}[1]{\beta_{#1}}
\newcommand{\indnonconf}[1]{\delta_{#1}}
\newcommand{\indnonconfel}[1]{\tilde\delta_{#1}}
\newcommand{\indellipt}[1]{\varepsilon_{#1}}
\newcommand{\indelliptplus}[1]{{\varepsilon^{\smash{\scriptscriptstyle{\boldsymbol+}}}_{#1-1}}}
\newcommand{\parest}[1]{\eta_{\mathrm{p},#1}}
\newcommand{\ellest}[1]{\eta_{\mathrm{e},#1}}
\newcommand{\difflipA}[2]{\sqrt{\mat a(#2)\inverse{\mat a(t_{#1})}}
      -\sqrt{\inverse{\mat a(#2)\!}\!\!\mat a(t_{#1})}}
\newcommand{\difflipa}[1]{\lambda_{\mat a,#1}}
\renewcommand{\C}[1]{\ensuremath{C_{\operatorname{#1}}}}
\newcommand{\data}[1]{
  \ensuremath{
    \fraclpf{\interpol{#1}\U^{#1-1}-\U^{#1-1}}{\tau_n}
    +
    f^{#1}-f
    }
}

\newcommand{\EI}{\operatorname{EI}}

\usepackage{pdfsync}
\swapnumbers
\theoremstyle{plain}
\newtheorem{Lemma}[subsection]{Lemma}
\newtheorem{Theorem}[subsection]{Theorem}

\newtheorem{Problem}[subsection]{Problem}
\theoremstyle{note}
\newtheorem{Definition}[subsection]{Definition}
\newtheorem{Remark}[subsection]{Remark}
\numberwithin{equation}{section}
\newcommand{\fulltitle}{A posteriori error control for discontinuous {Galerkin} methods for parabolic problems\xspace}
\newcommand{\headtitle}{A POSTERIORI CONTROL IN PARABOLIC DISCONTINUOUS GALERKIN\xspace}
\newcommand{\authorone}{Emmanuil~H.~Georgoulis\xspace}
\newcommand{\addressone}{
  Department of Mathematics,
  University of Leicester,
  University Road,
  Leicester,
  LE1 7RH, United Kingdom\xspace}
\newcommand{\emailone}{\linkedemail{Emmanuil.Georgoulis@mcs.le.ac.uk}\xspace}
\newcommand{\authortwo}{Omar~Lakkis\xspace}
\newcommand{\addresstwo}{
  Department of Mathematics,
  University of Sussex,
  Falmer near Brighton,
  East Sussex,
  {GB-BN1~9RF}, England UK\xspace}
\newcommand{\emailtwo}{\linkedemail{o.lakkis@sussex.ac.uk}\xspace}
\newcommand{\authorthree}{Juha~M.~Virtanen\xspace}
\newcommand{\addressthree}{
  Department of Mathematics,
  University of Leicester,
  University Road,
  Leicester,
  LE1 7RH, United Kingdom\xspace}
\newcommand{\emailthree}{\linkedemail{jmv8@leicester.ac.uk}\xspace}
\newcommand{\amssubjectclassification}{65M15, 65M60, 65N30\xspace}
\title{\fulltitle}
{\title[\headtitle]{\fulltitle}
\author{\authorone}
\address{
  \authorone\newline
  \addressone
}
\email{\emailone}
\author{\authortwo}
\address{
  \authortwo\newline
  \addresstwo
}
\email{\emailtwo}
\author{\authorthree}
\address{
  \authorthree\newline
  \addressthree
}
\email{\emailthree}
\subjclass[2000]{\amssubjectclassification}
}
\begin{document}
\frenchspacing
\maketitle
\renewcommand{\thefootnote}{\fnsymbol{footnote}}
\footnotetext[1]{\addressone, \emailone, \emailthree}
\footnotetext[4]{\addresstwo, \emailtwo}
\renewcommand{\thefootnote}{\arabic{footnote}}
\begin{abstract}
We derive energy-norm \aposteriori error bounds for an Euler
time-stepping method combined with various spatial discontinuous
Galerkin schemes for linear parabolic problems.  For accessibility, we
address first the spatially semidiscrete case, and then move to the
fully discrete scheme by introducing the implicit Euler time-stepping.
All results are presented in an abstract setting and then illustrated
with particular applications.  This enables the error bounds to hold
for a variety of discontinuous Galerkin methods, provided that
energy-norm \aposteriori error bounds for the corresponding elliptic
problem are available.  To illustrate the method, we apply it to the
interior penalty discontinuous Galerkin method, which requires the
derivation of novel \aposteriori error bounds.  For the analysis of
the time-dependent problems we use the elliptic reconstruction
technique and we deal with the nonconforming part of the error by
deriving appropriate computable \aposteriori bounds for it.  
\changes{
  We illustrate the theory with a series of numerical experiments
  indicating the reliability and efficiency of the derived
  \aposteriori estimates.
}
\end{abstract}
%
%
\section{Introduction}
\label{sec:introduction}
Adaptive methods for partial differential equations (PDE's) of
evolution type have become a staple in improving the efficiency in
large scale computations.  Since the 1980's many adaptive methods have
been increasingly based on \emph{\aposteriori error estimates}, which
provide a sound mathematical case for \emph{adaptive mesh refinement},
which can be decomposed in spatially and temporally local \emph{error indicators}.
In the context of
parabolic equations, \aposteriori error estimates have been derived
for various norms since early 1990's
\cite{eriksson-johnson:1,picasso:98}.  Inspired by the milestones set
recently for the mathematical theory of convergence for adaptive
methods in elliptic problems
\cite{morin-nochetto-siebert:2002:review,binev-dahmen-devore:04,
cascon-kreuzer-nochetto-siebert:08}, there has been a recent push for
similar results for parabolic problems calling to a closer
understanding of \aposteriori error estimates \cite[e.g.]{chen-jia:03,
verfuerth:03, bergam-bernardi-mghazli:05, lakkis-makridakis:06}.  Most
results in this area cover simple time-stepping schemes and a
conforming space discretization.  The extant literature on a
posteriori error control for nonconforming spatial methods can be
grouped in a handful of works
\cite{sun-wheeler:05,ern-proft:05,yang-chen:06,nicaise-soualem:05,
chenyanping-yangjiming:07}.  In \cite{sun-wheeler:05} \aposteriori
$\leb2(\sobh1)$-norm error bounds for a spatially semidiscrete method via
interior penalty discontinuous Galerkin (\ipdg) methods are derived
and used heuristically in the implementation of the fully discrete
scheme.  In \cite{ern-proft:05,yang-chen:06} $\leb2(\leb2)$-norm error
bounds for \ipdg are obtained using duality techniques, while in
\cite{nicaise-soualem:05}, \aposteriori error bounds are presented for
a fully discrete method consisting of a backward Euler time-stepping
and linear Crouzeix--Raviart elements in space.  Note that none of the
papers in the literature, to our knowledge, cover the case of
\aposteriori energy-norm error bounds for fully discrete schemes with
discontinuous Galerkin methods, which is the chief objective of our paper.

\emph{Discontinuous Galerkin} (DG) methods are an important family of
nonconforming finite element methods for elliptic, parabolic and
hyperbolic problems dating back to 1970's and early 1980's
\cite{nitsche:71,reed-hill:73,baker:77:non-conf,wheeler:78,arnold:82}.  DG methods have
undergone substantial development in the recent years \cite[e.g., and
references therein]{cockburn-karniadakis-shu:01:dg_book,arnold-brezzi-cockburn:02:unified,riviere-wheeler-girault:99,riviere-wheeler:00,houston-schwab-suli:02}.
The practical interest in DG methods owes to their flexibility in mesh
design and adaptivity, in that they cover meshes with hanging nodes and/or
locally varying polynomial degrees.  DG methods are thus ideally
suited for $hp$-adaptivity and provide good local conservation
properties of the state variable.  Moreover, in DG methods the local
elemental bases can be chosen freely for the absence of interelement
continuity requirements, yielding very sparse---in many cases even
diagonal---mass matrices even with high precision quadrature. Note
also that DG methods are popular due to their very good stability properties
in transport- or convection-dominated problems
\cite{cockburn-karniadakis-shu:01:dg_book}; the \aposteriori error
analysis of convection-dominated problems is, however, beyond the scope of our
study and we concentrate on diffusion-only parabolic equations.

Our main results are \aposteriori error bounds in the \emph{energy
norm} for a family of fully discrete approximations of the following
PDE problem---in \secref{sec:preliminaries} we gather the functional
analysis notation and background.
\begin{Problem}[linear parabolic boundary-initial value problem]
  \label{pbm:linear-diffusion}
  Given an open (possibly curvilinear) polygonal domain $\W\subseteq\R
  d$, $d=2,3$, a real number $T>0$, two (generalized)
  functions
  \begin{equation}
    f\in\leb\infty(0,T;\leb2(\W))
    \AND
    \mat a\in\leb\infty(\W\times\opinter0T)^{d\times d},
  \end{equation}
  such that $\mat a(x,t)$ is symmetric positive definite for almost
  all $(x,t)\in\W\times\clinter0T$, find a function
  $u\in\leb2(0,T;\honezw)$,
  \begin{equation}
    \pdt{u}\in\leb\infty(0,T;\sobh{-1}(\W))
  \end{equation}
  and such that
  \begin{equation}
    \label{pde}
    \begin{gathered}
      \pdt u-\div(\mat a\grad u)=f
      \text{ on }\W\times\opclinter0T,
      \\
      u(0)=u_0\text{ on }\W,
      \AND
      \restriction{u}{\boundary\W}(t)=0,
      \text{ for }t\in\opclinter0T.
    \end{gathered}
  \end{equation}
\end{Problem}
In \secref{sec:preliminaries} we propose a class of numerical methods
for solving this problem.  These methods consist in a backward Euler
time-stepping scheme in combination with various choices of spatial DG
methods.  Our emphasis is on the widely applied \ipdg method
\cite{arnold:82,riviere-wheeler-girault:99,houston-schwab-suli:02}.

We consider the notation of Problem \ref{pbm:linear-diffusion} to be
valid throughout the paper and $u$ denotes the solution of problem
(\ref{pde}).  Although the assumption $f\in\leb\infty(0,T;\leb2(\W))$
may be weakened---provided \aposteriori error estimates for the
corresponding spatial finite element method can be obtained for such
weak data---we refrain from doing it for simplicity's sake.  In fact,
we consider $f$ to be piecewise continuous in time with a finite
number of time-discontinuities and with the implied constraints on the
time partition, to be discussed in \secref{sec:fully-discrete-solution}.
The matrix-valued function $\mat a(\cdot,t)$, for each
$t\in\opinter0T$ is allowed to have jump discontinuities; the set of
spatial discontinuities of $\mat a$ will be considered to be aligned
with the finite element meshes.  For simplicity, we shall assume that
$\mat a$ is continuous in time, but a finite number of discontinuities
can be accounted for easily, as long as these occur at the points of
the time partition in the fully discrete scheme.  The PDE (\ref{pde})
is assumed to be uniformly elliptic in the sense that the supremum and
the infimum of the set
\begin{equation}
  \ensemble{\fracl{\qp{\mat a(x,t)\vec\zeta}\inner\vec\zeta}
    {\abs{\vec\zeta}^2}}
    {\vec\zeta\in\R d,(x,t)\in\W\times\opclinter0T}
\end{equation}
are both positive real numbers.  As for the boundary values, we remark
that our approach can be appropriately modified in order to extend
homogeneous to general time-dependent Dirichlet boundary values.
Under the assumptions made so far, we have that the solution to
(\ref{pde}) exists and satisfies $u\in\cont0(0,T;\sobhz1(\W))$ and
$\pdt u\in\leb2(0,T;\leb2(\W))$
\cite{ladyzenskaja-solonnikov-uralceva:book}.

In line with a unified approach to \aposteriori error analysis for
elliptic-problem DG methods \cite{ainsworth:05:synthesis,ainsworth:07,
carstensen-gudi-jensen:08} our discussion will be presented first in
an abstract setting. Our results are then shown to be applicable to a
wide class of DG methods \emph{provided that \aposteriori error bounds
for the corresponding steady-state problem are available}.

We stress that, although we focus on \aposteriori error bounds for
spatial DG methods, our abstract results can be applied to a wider
class of nonconforming methods (other than DG methods), provided they
satisfy certain requirements. More specifically, given a particular
nonconforming finite element space $\fesh$, assume that:
\begin{LetterList}
  \item
    for each $Z\in\fesh$ it is possible to decompose it as
    \begin{equation}
      Z=Z\conpart+Z\dispart\text{ such that }Z\conpart\in\honezw\cap\fesh,
    \end{equation}
    where $Z\conpart$ and $Z\dispart$ are called $Z$'s
    \emph{conforming part} and the \emph{nonconforming part},
    respectively.  This decomposition is an analytic device and is not
    needed for computational purposes.  The only requirement on this
    decomposition is the ability to quantify certain norms of
    $Z\dispart$ in terms of $Z$, as found in the literature
    \cite[e.g.]{becker-hansbo-larson:03, karakashian-pascal:03,houston-schoetzau-wihler:07}, as well as our Lemma~\ref{h_oswald}.
  \item
    Given a function $z\in\honezw$, and let $Z\in\fesh$ be the
    corresponding Ritz-projection via the finite element method, it is
    possible to bound the norm of the \emph{error} $Z-z$, using
    \aposteriori error estimators for the steady state problem.
\end{LetterList}

A key tool in our \aposteriori error analysis is the \emph{elliptic
  reconstruction technique} \cite{makridakis-nochetto:03}.  Roughly
speaking, the elliptic reconstruction technique, as far as energy
estimates are concerned, allows to neatly separate the time
discretization analysis form the spatial one.  This technique, which
has been adapted to tackle fully-discrete schemes via energy methods
for conforming methods \cite{lakkis-makridakis:06}, is extended in
this work to the nonconforming setting, to all methods that meet the
two requirements above.  Briefly said, the idea of elliptic
reconstruction---denoting by $u$ the solution of (\ref{pde}) and by
$U$ that of the discrete problem---consists in building an auxiliary
function $\w$, called the \emph{elliptic reconstruction} of $U$, which
satisfies two key properties: (a) a PDE-like relation binds the
\emph{parabolic error} $u-w$ with data quantities only involving $w-U$
and the problem's data, $f$, $\mat a$, and $u_0$, (b) the function $U$
is the Ritz projection of $\w$ onto $\fesh$.  Note that $\w$ is an
analysis-only device that, despite its name, it is not a computable
object. Fortunately, computing $\w$ is not needed in practice, as it
does not appear in the resulting \aposteriori bounds.

We believe that it is possible to obtain similar \aposteriori error
estimates, for each single method at hand, by working directly, i.e.,
without using an elliptic reconstruction technique, but this will
inevitably lead to further complications which may render the analysis
quite involved, especially for the fully discrete scheme.  This
prejudice of ours is testified by the somewhat surprising lack of
previous rigorous results in the literature.  Finally, we point out
that it is possible to follow a similar approach to ours in order to
derive \aposteriori error estimates in lower order functional spaces
such as $\leb\infty(0,T;\leb2(\W))$.

We remark that new estimators arise in the derivation of fully
discrete \aposteriori error bounds, due to the time-dependent
diffusion tensor considered in this work, compared to
\cite{lakkis-makridakis:06} where only time-independent diffusion
coefficients are addressed.

The following is an outline of this article.  After introducing the
notation and the method in \secref{sec:preliminaries}, the elliptic
reconstruction is used to develop an abstract framework for spatially
semidiscrete schemes in \secref{sec:semidiscrete} and their fully
discrete counterpart in \secref{sec:fully-discrete-bounds}.  The
actual error estimators for each particular method are then
consequences of our abstract framework and specific elliptic error
estimators such as the ones presented in
\cite{becker-hansbo-larson:03,
  karakashian-pascal:03,bustinza-gatica-cockburn:05,
  houston-schoetzau-wihler:07,houston-suli-wihler:08,ern-stephansen:08}.
Moreover, in \secref{sec:various-DG} we prove \aposteriori bounds for the
corresponding steady state problem of (\ref{pbm:linear-diffusion}) for
\ipdg, thus extending existing results \cite{becker-hansbo-larson:03,
  karakashian-pascal:03,houston-schoetzau-wihler:07,houston-suli-wihler:08}
to the case of general (non-diagonal) diffusion tensor, with minimal
regularity assumptions on the exact solution
\cite[cf.]{ern-stephansen:08}.  These \aposteriori bounds are then
combined with the general framework presented in in
\secref{sec:semidiscrete} and \secref{sec:fully-discrete-bounds} to
deduce fully computable bounds for the DG-approximation error of the
parabolic problem.  \changes{Last in \secref{sec:computer-experiments}
  we summarize results from computer experiments aimed at exhibiting
  the reliability (derived theoretically) and efficiency of the error
  estimators in the special case of the \ipdg method.}
\section{Preliminaries}
\label{sec:preliminaries}
\subsection{Functional analysis tools}
\label{sec:mod}
Given an open subset $\omega\subseteq\R d$, we denote by
$L^p(\omega)$, $1\le p\le \infty$, the Lebesgue spaces of functions
with summable $p$-powers on $\omega$.  The corresponding norms
$\|\cdot\|_{L^p(\omega)}$; the norm of $\leb2(\omega)$ will be denoted
by $\ltwo{\cdot}{\omega}$ for brevity; by $\langle
\cdot,\cdot\rangle_{\omega}$ we write the standard $\leb2$-inner product
on $\omega$.  When $\omega=\W$ we omit the subindex.

We denote by $\sobh s(\omega)$, the standard Hilbert Sobolev space of
index $s\in\reals$; $\sobh1_0(\omega)$ signifies the subspace of
$\sobh1(\omega)$ of functions with vanishing trace on the boundary
$\partial\omega$.  The Poincaré--Friedrichs inequality
\begin{equation}
  \label{eqn:Poincare-Friedichs}
  \ltwo v\W\leq\C{PF}\ltwo{\grad v}\W
  \text{ for }v\in\sobhz1(\W)
\end{equation}
turns the $\sobh1(\W)$ seminorm into a norm on $\sobhz1(\W)$.
We consider thus $\ltwo{\grad\cdot}\W$ to be the norm on
$\sobhz1(\W)$.  We will use also $\sobh{-1}(\W)$, the dual
space of $\sobhz1(\W)$, equipped with the duality brackets
$\duality \cdot\cdot$.  Namely, if $f\in\sobh{-1}(\W)$ then
for each $\phi\in\sobhz1(\W)$ its value on $\phi$ is denoted by
$\duality g\phi$ which coincides with $\ltwop g\phi$ if
$g\in\leb2(\omega)$.  Thus the norm of $g$ is given by
\begin{equation}
  \Norm g_{\sobh{-1}(\W)}
  :=\sup_{\phi\in\sobhz1(\W)\take\setof0}
  \frac{\ltwop g\phi}{\Norm{\grad\phi}}.
\end{equation}

The duality pairing $\pairing{\sobh{-1}}{\sobhz1}$ allows us to
define, for each fixed $t\in\opclinter0T$ the \emph{elliptic operator}
$\funk{\opera(t)}{\sobhz1(\W)}{\sobh{-1}(\W)}$ where
\begin{equation}
  \duality{\opera(t) v}\phi:=\ltwop{\mat a(t)\grad v}{\grad \phi}
  \qp{=\int_\W\qp{\mat a(x,t)\grad v(x)}\inner{\grad \phi(x)}\d x}
\end{equation}
$\text{ for }\phi,v\in\sobhz1(\W)$. Here, and
throughout the paper, we use the shorthand $f(t)=f(\cdot,t)$, for a
function $\funk f{[0,T]\times\W}\reals$.  Note that thanks to the
uniform parabolic assumption given by (\ref{growth}) and the
Lax--Milgram Theorem, the definition of the operator $\opera(t)$ is
well defined and yields an isomorphism between $\sobh{-1}(\W)$ and
$\sobhz1(\W)$ \cite{evans:PDE}. In other words, the operator
$\opera(t)$ induces a bounded coercive bilinear form
\begin{equation}
  \label{eqn:bilinear-form-on-H10}
  (v,w)\in\sobhz1(\W)\times\sobhz1(\W)\mapsto\duality{\opera(t)v}w\in\reals,
\end{equation}
which we will extend later to a larger nonconforming space.  \emph{We
  stress from the outset that although the bilinear form in
  (\ref{eqn:bilinear-form-on-H10}) will be extended later to larger
  spaces, the operator $\opera(t)$ will not and it will only act on
  $\sobhz1(\W)$ functions throughout the discussion.}

For $1\le p\le +\infty$, we also define the spaces
$L^p(0,T,X)$, with $X$ being a real Banach space with norm
$\|\cdot\|_X$, consisting of all measurable functions $v: [0,T]\to X$,
for which
\begin{equation}
  \begin{gathered}
    \|v\|_{L^p(0,T;X)}
    :=\Big(\int_0^T \|v(t)\|_{X}^p\ud t\Big)^{1/p}<+\infty,
    \quad \text{for}\quad 1\le p< +\infty, \\
    \|v\|_{L^{\infty}(0,T;X)}:=\esssup_{0\le t\le
      T}\|v(t)\|_{X}<+\infty,\quad \text{for}\quad  p = +\infty.
  \end{gathered}
\end{equation}
Finally, we denote by $C(0,T;X)$ the space of continuous
functions $v: [0,T]\to X$ with norm $\|v\|_{C(0,T;X)}:=\max_{0\le
t\le{T}}\|v(t)\|_{X}<+\infty$.
\hyphenation{pa-ra-me-tri-zed}
\subsection{Finite element spaces}
Let $\mathscript{T}$ be a subdivision of $\W$ into disjoint open sets,
which we call elements.  We assume $\mathscript{T}$ to be parametrized
by mappings $F_{\kappa}$, for each $\el$, where
$F_{\kappa}:\hat{\kappa}\to\kappa$ is a diffeomorphism and
$\hat{\kappa}$ is the reference element or reference square. The above
mappings are such that $\bar{\W}=\cup_{\el}\bar{\kappa}$.  We often
use the word \emph{mesh} for subdivision, and we say that a mesh is
\emph{regular} if it has no hanging nodes; otherwise the mesh is
\emph{irregular}.  Unless otherwise stated, we allow the mesh to be
$1$-irregular, i.e., for $d=2$, there is at most one hanging node per
edge, typically its center; for $d=3$ a corresponding concept is
available.

For an integer $p\geq1$, we denote by $\mathscript P_p(\hat{\kappa})$,
the set of all polynomials on $\hat{\kappa}$ of degree $p$, if
$\hat{\kappa}$ is the reference simplex, or of degree $p$ in each
coordinate direction, if $\hat{\kappa}$ is the reference cube.  We
consider the discontinuous Galerkin finite element space
\begin{equation}
  \label{eq:FEM-spc}
  S=S^p(\mathscript{T}):=\{v\in \leb2(\W):v_{\kappa}\circ
  F_{\kappa}
  \in \mathscript{P}_p(\hat{\kappa}),\,\el\}.
\end{equation}

By $\Ga$ we denote the union of all sides of the elements of the
subdivision $\mathscript{T}$ (including the boundary sides). We think
of $\Ga$ as the union of two disjoint subsets
$\Ga=\Ga_{\partial}\cup\Ga_{\dint}$, where
$\Ga_{\partial}$ is the union of all boundary sides.

Let two elements $\kappa^+,\kappa^-\in\cT$ have a common a side
$e:=\bar{\kappa}^+\cap\bar{\kappa}^-\subset\Ga_{\dint}$.  Define
the outward normal unit vectors $\vec n^+$ and $\vec n^-$ on $e$
corresponding to $\partial\kappa^+$ and $\partial\kappa^-$,
respectively. For functions $q:\W\to\mathbb{R}$ and
$\vec\phi:\W\to\mathbb{R}^d$ uniformly continuous on each of
$\kappa^\pm$, but possibly discontinuous across $e$, we define the
following quantities. For $q^+:=q|_{\partial\kappa^+}$,
$q^-:=q|_{\partial\kappa^-}$ and
$\vec\phi^+:=\vec\phi|_{\partial\kappa^+}$,
$\vec\phi^-:=\vec\phi|_{\partial\kappa^-}$, we set
\begin{equation}
  \begin{aligned}
    \mean{q}_e&:=\ha(q^+ + q^-),
    \quad
    &\mean{\vec\phi}_e&:=\ha(\vec\phi^+ + \vec\phi^-),
    \\
    \jump{q}_e&:=q^+\vec n^++q^-\vec n^-,
    \quad
    &\jump{\vec\phi}_e&:=\vec\phi^+\cdot \vec n^++\vec\phi^-\cdot \vec n^-;
  \end{aligned}
\end{equation}
if $e$ is a boundary side ($e\subset\Ga_{\partial}$) these
definitions are modified to
\begin{equation}
  \mean{q}_e:=q^+,\
  \mean{\vec\phi}_e:=\vec\phi^+ ,\
  \jump{q}_e:=q^+\vec n^+,\
  \jump{\vec\phi}_e:=\vec\phi^+\cdot \vec n^+.
\end{equation}
We introduce the \emph{mesh-size} as the function $h:\W\to
\mathbb{R}$, by $h(x)=\diam\kappa$, if $x\in \kappa$ and $h(x)=
\mean{h}$, if $x\in \Ga$.  The \emph{shape-regularity} of the
subdivision $\cT$ is defined as
\begin{equation}
  \mu(\cT):=\sup_{\kappa\in\cT}\frac{h_\kappa}{r_\kappa},
\end{equation}
where $r_\kappa$ is the radius of the largest ball that fits
entirely in $\kappa$.

We shall use the gradient's \emph{regular part operator}, $\dgrad$,
of an elementwise differentiable function $v$ defined by
\begin{equation}
  (\dgrad v)|_{\kappa}:=\nabla (v|_{\kappa})\text{ for }\el.
\end{equation}
Note that the full distributional gradient, $\grad v$, consists of an
extra term taking into account the jumps of $v$ across the edges (with
a $-$ sign for historic reasons):
\begin{equation}
  \grad v=\dgrad v-\jump v \delta_{\Ga_\dint},
\end{equation}
with $\delta_{\Ga_\dint}$ denoting the Dirac distribution on the interior skeleton $\Ga_{\dint}$.
Finally, we consider some shorthand notation for quantities involving
the diffusion tensor $a$.  In particular, we define the elementwise
constant functions $\amink,\amaxk:\W\times [0,T]\to\mathbb{R}$ by
\begin{equation}
  \amaxk(\cdot,t)|_{\kappa}:=\||\sqrt{a}(\cdot,t)|_2\|_{L^{\infty}(\kappa)}^{2}
  \AND
  \amink(\cdot,t)|_{\kappa}:=\||(\sqrt{a}(\cdot,t))^{-1}|_2\|_{L^{\infty}(\kappa)}^{-2},
\end{equation}
for $\el$, and $\amaxk=\mean{\amaxk}$, $\amink=(\mean{1/\amink})^{-1}$, on
$\Ga$, where $|\cdot|_2$ denotes the Euclidean-induced matrix norm.
Finally, let
\begin{equation}\label{growth}
  \amax(t):=\max_{x\in\W}\amaxk(x,t)
  \AND
  \amin(t)=\min_{x\in\W}\amink(x,t).
\end{equation}
\subsection{Spatial DG discretization}
\label{dg_heat}
Introduce the \emph{DG space}
$\mathscript{S}:=S+\sobhz1(\W)$, and a corresponding
\emph{DG bilinear} form $B:\mathscript{S}\times
\mathscript{S}\to \mathbb{R}$, which we assume to be an extension of
the bilinear form defined by (\ref{eqn:bilinear-form-on-H10}), viz.,
\begin{equation}
  \label{cons}
  B(t;v,z)=\duality{\opera(t)v}z
  \text{ for all }v,z\in \sobhz1(\W),t\in\opclinter0T.
\end{equation}
Though $B$ is time-dependent, we do not write it explicitly in the
semidiscrete case and omit the $t$.

The space $\cS$ is equipped with a \emph{DG norm}, denoted
$\ndg{\cdot}$ and depending on the method at hand, which extends the
energy norm, i.e.,
\begin{equation}
  \label{norm_cons}
  \ndg{v}=\ltwo{\sqrt{a(\cdot,t)}\nabla v}{}
  \quad\text{for all}\ v\in \sobhz1(\W),
\end{equation}
for $t\in[0,T]$.  Also here, the norm is time-dependent, but this
dependence is not explicitly written.  A norm equivalence between the
energy norm $\ndg{\cdot}$ and $\ltwo{\sqrt{a(\cdot,t)}\nabla \cdot}{}$
in $\sobhz1(\W)$, uniformly with respect to $t$, suffices for all
the bounds presented below to hold, modulo a multiplicative constant;
but, we eschew this much generality for clarity's sake.

The semidiscrete DG method in space for problem (\ref{pde}), reads as
follows:
 \begin{equation}
   \label{dg_semi}
   \begin{split}
     &\text{Find $U\in\cont{0,1}(0,T;\fesh)$ such that}
     \\
     &\langle \partial_t
     U,V\rangle+B(U,V)=\langle f,V\rangle
     \quad\text{ for }V\in S, t\in[0,T].
   \end{split}
\end{equation}

We stress that these assumptions are satisfied by many DG methods for
second order elliptic problems available in the literature, possibly
by using inconsistent formulations
\cite{arnold-brezzi-cockburn:02:unified}. Moreover, (\ref{cons}) and
(\ref{norm_cons}) are satisfied by \ipdg (along with the corresponding
energy norm) considered below as a paradigm.

Assumption (\ref{cons}) implies the consistency of the bilinear form
$B$ on $\sobhz1(\W)$, i.e.,
\begin{equation}
  \label{consistency}
  \langle \partial_t u,v\rangle+B(u,v)=\langle f,v\rangle,\quad
  \text{for all}\ v\in \sobhz1(\W),
\end{equation}
where $u$ is the exact (weak) solution to the initial-boundary
value problem (\ref{pde}).  Note that this is the same as writing
\begin{equation}
  \pdt u+\opera u=f.
\end{equation}
\subsection{Fully discrete solution}
\label{sec:fully-discrete-solution}
To further discretize in time, consider an increasing time partition
$\seq{\tn}_{\rangefromto n0N}$, and the corresponding time-steps
$\taun=\tn-\tno$, for $\rangefromto n1N$. For each $\rangefromto n0N$, we consider that $\fes n$
is a DG finite element space of fixed degree $p$ built on a partition
$\cT_n$, which may be different from $\cT_{n-1}$ when $n\geq1$.  In
\secref{sec:mesh-interaction-decomposition}, we will say more about the
sequence of meshes and the compatibility relations among them.

Let $\fn(x):=f(x,\tn)$ and let $\UZERO$ be the projection (or an
interpolation) of $\uzero$ onto the finite element space $\fes 0$.  We
say that $\seq{\un}_{\rangefromto n0N}$ is a \emph{fully discrete
solution} of (\ref{pde}) if, for each $\rangefromto n1N$ we have that
$\un\in\fes n$ satisfies
\begin{equation}
  \label{eqn:fully-discrete-Euler-pure}
  \ltwop{\fraclpf{\un-\uno}{\taun}}{V}
  +
  \bbilp n\un V=\ltwop{\fn}{V}\Forall V\in\fes n,
\end{equation}

Since the elliptic operator $\opera$ (and the bilinear form
$\bilform$) depend on time, in the fully discrete setting, we denote
their value at time $t$ by $\opera(t)$ and $\bilform(t)$,
respectively and when $t=\tn$ we take $\opera^n:=\opera(\tn)$
and $B^n:=B(\tn)$.

Noting that the term $\uno$ can be replaced by $\lproj n\uno$, where
$\funk{\lproj n}{\leb2(\W)}{\fes n}$ is the orthogonal projection, we
consider a slightly more general situation where $\lproj n\uno$ in
(\ref{eqn:fully-discrete-Euler-pure}) is replaced by $\interpol
n\uno$;  here $\funk{\interpol n}{\fes{n-1}}{\fes n}$ is a general
\emph{data transfer} operator, depending on the particular
implementation.  The operator $\interpol n$ may coincide with $\lproj
n$, but it may be an interpolation operator for example.
The general fully discrete Euler scheme then reads
\begin{equation}
  \label{eqn:fully-discrete-Euler}
  \ltwop{\fraclpf{\un-\interpol n\uno}{\taun}}{V}
  +
  \bbilp n\un V=\ltwop{\fn}{V}\Forall V\in\fes n.
\end{equation}

We have taken $\fn=f(\tn)$, but we could take a more general
approximation than $f(\tn)$, for example, a good choice is also given
by $\fn:=\int_\tno^\tn f(s)\d s$, for which a suitable modification of
our arguments leads to similar results.
\section{Abstract \aposteriori bounds for the semidiscrete problem}
\label{sec:semidiscrete}
We derive next an abstract a posteriori error bound for the quantity
\begin{equation}
  \ltwo{u-U}{\leb2(0,T;\mathscript{S})}
  :=\Big(\int_0^T\ndg{u(t,\cdot)-U(t,\cdot)}^2\Big)^{1/2},
\end{equation}
where $\ndg{\cdot}$ denotes the appropriate (space) energy norm.

In the a posteriori error analysis below, we shall make use of the
idea of elliptic reconstruction operators introduced in
\cite{makridakis-nochetto:03} for the semidiscrete problem and
extended to fully discrete (conforming-in-space) methods in \cite{lakkis-makridakis:06}.
\begin{Definition}[elliptic reconstruction and discrete operator]
\label{elliptic_recon}
Let $U$ be the (semidiscrete) DG solution to the problem
(\ref{dg_semi}). We define the \emph{elliptic reconstruction}
$w\in \sobhz1(\W)$ of $U$ to be the
solution of the elliptic problem
\begin{equation}\label{ell_rec}
  B(w, v) = \langle AU-\Pg f +f, v\rangle\quad \Forall v\in
  \sobhz1(\W),
\end{equation}
where $\Pg:\leb2(\W)\to S$ denotes the
orthogonal $\leb2$-projection on the finite element space $S$, and
$A:S\to S$ denotes the discrete DG operator defined by
\begin{equation}\label{ell_rec2}
  \langle A Z, V \rangle = B(Z,V)
  \Forall V\in S,
\end{equation}
for each $Z\in S$.  Note that this is valid on $\clinter0T$.
\end{Definition}
\begin{Remark}[the role of the elliptic reconstruction]
  \label{remark_recon}
  The elliptic reconstruction is well defined. Indeed, $AU\in S$ is the
  unique $\leb2$-Riesz representation of a linear functional on the
  finite dimensional space $\fesh$ and the existence and uniqueness of
  (weak) solution of (\ref{ell_rec}), with data
  $AU-\lproj{}f+f\in\leb2(\W)$, follows from the Lax--Milgram Theorem in
  view of (\ref{cons}).

  The key property of $w$ is that the DG solution $U$ of the semidiscrete
  time-dependent problem (\ref{dg_semi}) is also the DG solution of
  the steady-state boundary-value problem (\ref{ell_rec}). Indeed, let
  $W\in S$ be the DG-approximation to $w$, defined by the finite
  dimensional linear system
  \begin{equation}
    B(W,V)=\langle AU-\Pg f +f, V\rangle,
  \end{equation}
  for all $V\in S$, which implies $B(W,V)=\langle AU, V\rangle=B(U,V)$
  for all $V\in S$, i.e., $W=U$.
\end{Remark}
\begin{Definition}[error, elliptic and parabolic parts]
  We shall decompose the error as follows:
  \begin{equation}\label{splitting_semi}
    e:=U-u=\rho-\epsilon,\ \text{where}\ \epsilon:=w-U,\ \text{and}\
    \rho:=w-u,
  \end{equation}
  where $w=w(t)$ denotes the elliptic reconstruction of $U=U(t)$ at
  time $t\in[0,T]$.  We call $\epsi$ the \emph{elliptic error} and
  $\rho$ the \emph{parabolic error}.
\end{Definition}
\begin{Lemma}[semidiscrete error relation]
\label{energy_arg}
Let $u$ be the solution of Problem \ref{pbm:linear-diffusion}, $U$
denote the solution of the DG scheme (\ref{dg_semi}). Then, we have
\begin{equation}\label{en_ar}
  \langle \partial_t e,v\rangle+B(\rho,v)=0\Forall v\in \sobhz1(\W).
\end{equation}
\end{Lemma}
\begin{proof}
For each $v\in\sobhz1(\W)$ we have
\begin{equation}\label{rho-term}
  \begin{aligned}
    \langle \partial_t e,v\rangle+B(\rho,v)
    =& \langle \partial_t U,v\rangle + B(w,v) -\langle f,v\rangle
    \\
    =&
    \langle \partial_t U,v\rangle
    +\langle AU-\Pg f+f,v\rangle
    -\langle f,v\rangle
    \\
    =&
    \langle \partial_t U,\Pg v\rangle
    +\langle AU,\Pg v\rangle
    -\langle f,\Pg v\rangle
    =0,
  \end{aligned}
\end{equation}
where in the first equality we used (\ref{consistency}); in the second
and fourth equalities, we made use of Definition \ref{elliptic_recon};
in the third equality the properties of the orthogonal
$\leb2$-projection onto $S$ are used; finally, the last equality follows
from (\ref{dg_semi}).
\end{proof}
\begin{Definition}[conforming-nonconforming decomposition]
  In the theory developed below, we shall consider the decomposition
  of the DG solution $U\in S$ into conforming (continuous) and
  nonconforming (discontinuous) parts as follows
  \begin{equation}\label{decomp}
    U = U_c+U_d,
  \end{equation}
  where $U_c\in S_c:=\sobhz1(\W)\cap S$ and $U_d:=U-U_c\in S$. Note
  that at this point we do \emph{not} specify any particular
  decomposition, thus keeping the choice of such a decomposition at
  our disposal. Let
  \begin{equation}\label{eps_c}
    e_c:=e-U_d=U_c-u\in \sobhz1(\W), \quad\text{and}\quad
    \epsilon_c:=\epsilon+U_d=w-U_c\in \sobhz1(\W).
  \end{equation}
\end{Definition}
\begin{Theorem}[long-time \aposteriori error bound for DG]
\label{apost_semi_thm}
Let $u$ and $U$ be the exact weak solution of (\ref{pde}) and the DG
solution of the problem (\ref{dg_semi}), respectively.  Let $w$ be the
elliptic reconstruction of $U$ as in Definition \ref{elliptic_recon}.
Assuming (\ref{cons}) and (\ref{norm_cons}), and that a decomposition
of the form (\ref{decomp}) is available, then the following error
bound holds
\begin{equation}\label{apost_semi}
\begin{aligned}
  \ltwo{U-u}{\leb2(0,T;\mathscript{S})}\le &
  3\ltwo{w-U}{\leb2(0,T;\mathscript{S})}+2\ltwo{U_d}{\leb2(0,T;\mathscript{S})}
  \\
  &+2\ltwo{\partial_t U_d/\sqrt{\amin}}{{\leb2(0,T;\sobh{-1}(\W))}}
  +2\ltwo{u_0-U_c(0)}{}.
\end{aligned}
\end{equation}
\end{Theorem}
\begin{proof}
Set $v=e_c$ in (\ref{en_ar}); then, in view of (\ref{eps_c}), we have
\begin{equation}
  \label{basic_error} \langle \partial_t e_c,e_c\rangle
  +B(\rho,\rho)=B(\rho,\epsilon_c)-\langle\partial_t U_d,e_c\rangle.
\end{equation}

Recalling (\ref{cons}), (\ref{norm_cons}) and that $\rho,e_c\in
\sobhz1(\W)$ for every $t\in[0,T]$, using the Cauchy--Schwarz
inequality,
and the duality pairing $\pairing{\sobh{-1}}{\sobh1_0}$, we arrive to
\begin{equation}
  \label{semi_one}
  \ha\ud_t\ltwo{e_c}{}^2
  +\ndg{\rho}^2\le\ndg{\rho}\,\ndg{\epsilon_c}+\ltwo{\partial_t
    U_d}{\sobh{-1}(\W)}\ltwo{\nabla e_c}{}.
\end{equation}
Also, (\ref{eps_c}) implies
\begin{equation}
  \ltwo{e_c}{\sobh{1}(\W)}\le
  (\ndg{\epsilon_c}+\ndg{\rho})/\sqrt{\amin}.
\end{equation}
Setting $I_1:=\ndg{\epsilon_c}$,
$I_2:=\ltwo{\partial_t{U_d}/\sqrt{\amin}}{\sobh{-1}(\W)}$ in
(\ref{semi_one}) and rearranging lead to
\begin{equation}
  \ha\ud_t\ltwo{e_c}{}^2 +\ndg{\rho}^2\le\ndg{\rho}(I_1+I_2)+I_1I_2,
\end{equation}
which implies
\begin{equation}
  \ud_t\ltwo{e_c}{}^2 +\ndg{\rho}^2\le 4(I_1^2+I_2^2).
\end{equation}
Integration on $[0,T]$ and taking square roots yields
\begin{equation}
  \ltwo{\rho}{\leb2(0,T;\mathscript{S})} \le \ltwo{e_c(0)}{}
  +2\ltwo{\epsilon_c}{\leb2(0,T;\mathscript{S})}
  +2\ltwo{\partial_t U_d/\sqrt{\amin}}{{\leb2(0,T;\sobh{-1}(\W)})}.
\end{equation}
The assertion follows using triangle inequality on
(\ref{splitting_semi}) and (\ref{eps_c}).
\end{proof}
\begin{Theorem}[short-time \aposteriori error bound for DG]
\label{apost_semi_thm_l1}
Let the assumptions of Theorem \ref{apost_semi_thm} hold.
Then the following error
bound holds:
\begin{equation}\label{apost_semi_l1}
\begin{aligned}
  \ltwo{u-U}{\leb2(0,T;\mathscript{S})}\le &
  (\sqrt{3/2}+1)\ltwo{w-U}{\leb2(0,T;\mathscript{S})}
  +\sqrt{3/2}\ltwo{U_d}{\leb2(0,T;\mathscript{S})}
  \\
  &+2\ltwo{\partial_t U_d}{L^1(0,T;\leb2(\W))}
  +\sqrt{2}\ltwo{u_0-U_c(0)}{}.
\end{aligned}
\end{equation}
\end{Theorem}
\begin{proof}
  Let $T_0\in [0,T]$ be such that
  \begin{equation}
    \ltwo{e_c(T_0)}{}=\max_{0\le t\le T}\ltwo{e_c(t)}{}=:E_c.
  \end{equation}
  Then, (\ref{basic_error}) implies
  \begin{equation}
    \label{semi_two}
    \ha\ud_t\ltwo{e_c}{}^2
    +\ndg{\rho}^2\le\ndg{\rho}\,\ndg{\epsilon_c}
    +E_c\ltwo{\partial_t U_d}{},
  \end{equation}
  which, after integration on $[0,T_0]$, yields
  \begin{equation}
    \begin{aligned}
      \ha E_c^2
      +\ltwo{\rho}{\leb2(0,T_0;\mathscript{S})}^2
      \le &\ha \ltwo{e_c(0)}{}^2
      +\ltwo{\rho}{\leb2(0,T_0;\mathscript{S})}
      \ltwo{\epsilon_c}{\leb2(0,T_0;\mathscript{S})}\\
      &+E_c\ltwo{\partial_t U_d}{L^1(0,T_0;\leb2(\W))},
    \end{aligned}
  \end{equation}
  or
  \begin{equation}
  \label{semi_three}
  \frac{1}{4} E_c^2 \le \ha \ltwo{e_c(0)}{}^2
  + \frac{1}{4} \ltwo{\epsilon_c}{\leb2(0,T;\mathscript{S})}^2+\ltwo{\partial_t U_d}{L^1(0,T;\leb2(\W))}^2.
  \end{equation}
  Going back to (\ref{semi_two}), upon integration with respect to $t$
  between $[0,T]$, we obtain
  \begin{equation}
  \ha\ltwo{\rho}{\leb2(0,T;\mathscript{S})}^2\le \ha \ltwo{e_c(0)}{}^2
  + \ha\ltwo{\epsilon_c}{\leb2(0,T;\mathscript{S})}^2+\frac{1}{4}E_c^2+\ltwo{\partial_t U_d}{L^1(0,T;\leb2(\W))}^2,
  \end{equation}
  which, in conjunction with (\ref{semi_three}) gives
  \begin{equation}
  \ltwo{\rho}{\leb2(0,T;\mathscript{S})}^2\le  2\ltwo{e_c(0)}{}^2
  + \frac{3}{2}\ltwo{\epsilon_c}{\leb2(0,T;\mathscript{S})}^2+4\ltwo{\partial_t U_d}{L^1(0,T;\leb2(\W))}^2;
  \end{equation}
  the final bound now follows using the triangle inequality on (\ref{splitting_semi}) and (\ref{eps_c}).
\end{proof}
\begin{Remark}[long- versus short-time bounds]
We note that the crucial difference between bounds (\ref{apost_semi})
and (\ref{apost_semi_l1}) is that in the latter the $L^1$-accumulation
term $\ltwo{\partial_t U_d}{L^1(0,T;\mathscript{S})}$ is present; this
implies
\begin{equation}
  \ltwo{\partial_t U_d}{L^1(0,T;\leb2(\W))}\le
  \sqrt{T}  \ltwo{\partial_t U_d}{\leb2(0,T;\leb2(\W))},
\end{equation}
which may be preferable if $T<1$, but can be inefficient for long-time
integration. On the other hand, the corresponding term in
(\ref{apost_semi}) is
$\ltwo{\partial_t{U_d}/\sqrt{\amin}}{{\leb2(0,T;\sobh{-1}(\W)})}$,
which can be a bit less inefficient when the diffusion tensor $a$
varies substantially on $\W$.  \changes{We note, however, that it is
  possible to avoid dividing by the factor $\amin$, by equipping
  $\sobh{-1}(\W)$ with the dual norm of the energy norm
  (\ref{norm_cons}) in $\sobhz1(\W)$.  In practice, however, this
  improvement is relevant only if the dual norm is calculated
  explicitly~\cite{LakkisPryer:10}. Alternatively, one can apply a
  Poincaré--Friedrichs inequality to bound the dual norm by the
  $\leb2$-norm, which results into the reappearance of the factor
  $\amin$.}
\end{Remark}
\begin{Remark}[elliptic \aposteriori error estimates]
The bounds (\ref{apost_semi}) and (\ref{apost_semi_l1})
are not (yet) explicitly \aposteriori bounds:
$\ltwo{w-U}{\leb2(0,T;\mathscript{S})}$ still needs to be bounded by a
computable quantity.  To this end, given $g\in\leb2$, consider the
elliptic problem:
\begin{align}
  \label{elliptic_problem}
  &&\begin{split}
    &\text{find $z\in \sobhz1(\W)$ such that}
    \\
    &-\nabla\cdot (\mat a\nabla z) =g\
    \text{in }\W, \quad z=0\text{ on }\partial\W,
  \end{split}
  &
  \intertext{whose solution can be approximated by the following DG method:}
  &&\begin{split}\label{dg_elliptic}
    &\text{find $Z\in S$ such that}
    \\
    &B(Z,V)=\langle g,V\rangle\Forall V\in S.
  \end{split}
  &
\end{align}
If assume that an \emph{a posteriori estimator functional} $\cE$
exists, i.e.,
\begin{equation}
  \label{apost_ell}
  \ndg{z-Z}\le \mathscript{E}(Z,\mat a,g,\mathscript{T}),
\end{equation}
then we can \emph{computably} bound $\ltwo{w-U}{\leb2(0,T;\mathscript{S})}$ in
(\ref{apost_semi}) and (\ref{apost_semi_l1}) through
\begin{equation}
  \ltwo{w-U}{\leb2(0,T;\mathscript{S})}\le
  \Big(\int_0^T
  \mathscript{E}(U,\mat a, AU-\Pg f +f,\mathscript{T})^2\Big)^{1/2}.
\end{equation}
\Aposteriori bounds for various DG methods have been studied, under
different assumptions on data and admissible finite element spaces, by
many authors \cite{becker-hansbo-larson:03,karakashian-pascal:03,
houston-schoetzau-wihler:07,ainsworth:07,
houston-suli-wihler:08,ern-stephansen:08,carstensen-gudi-jensen:08}. Thus Theorems
\ref{apost_semi_thm} and \ref{apost_semi_thm_l1} can be applied to any
DG---and more generally to any non-conforming---method satisfying
(\ref{cons}) and (\ref{norm_cons}), and for which (\ref{apost_ell}) is
available.  The object of \secref{sec:various-DG} is to address this for \ipdg.
\end{Remark}
\section{\Aposteriori error bounds for the interior penalty DG method}
\label{sec:various-DG}
Here we extend the energy-norm a posteriori bounds for the family
\ipdg methods cf.\cite{becker-hansbo-larson:03,
karakashian-pascal:03, houston-schoetzau-wihler:07} for the Poisson
problem, to the case of the general diffusion problem
(\ref{elliptic_problem}). 
\changes{A similar analysis has recently appeared also in \cite{ern-stephansen:08}, while
a related DG method based on weighted averages for anisotropic and high-contrast diffusion problems can be found
in \cite{ern-stephansen-zunino:09}.}
We stress that our results can be
generalized as to allow for inhomogeneous or mixed boundary
conditions following \cite[resp.]{houston-schoetzau-wihler:07,karakashian-pascal:03}.


\begin{Definition}[\ipdg method]
  For $z,v\in \mathscript{S}$, the bilinear form
  $B:\mathscript{S}\times\mathscript{S}\to\mathbb{R}$ for the \ipdg
  method for the problem (\ref{elliptic_problem}) can be written as
  \begin{equation}
    \label{dg_elliptic_bilinear}
    B(z,v)
    :=
    \int_{\W}(\mat a\dgrad z)\cdot\dgrad v
    +\int_{\Ga}\big(\theta\mean{\mat a\Pg\nabla v}\cdot\jump{z}
    -\mean{\mat a\Pg\nabla z}\cdot\jump{v}+\sigma\jump{z}\cdot\jump{v}\big),
  \end{equation}
  for $\theta\in\{-1,0,1\}$, where $\Pg:[\leb2(\W)]^d\to S^d$
  denotes also the orthogonal $\leb2$-projection operator onto $S^d$,
  and the \emph{penalty function} $\sigma:\Ga\to \mathbb{R}$ is
  defined by
  \begin{equation}
    \label{eqn:def:penalty-sigma}
    \sigma:=\frac{C_{\mat a,\mu(\cT)}\mean{\amaxk}}h,
  \end{equation}
  where the constant $C_{\mat a,\mu(\cT)}>0$ depends on the
  shape-regularity of the mesh $\cT$ and on the smallest possible
  $C_{\mat a}>1$ such that
  \begin{equation}\label{lbv}
    C_{\mat a}^{-1}\le \frac{\amaxk|_{\kappa^+}}{\amaxk|_{\kappa^-}}\le C_{\mat a},
  \end{equation}
  for every pair of elements $\kappa^+$ and $\kappa^-$ sharing a common side.

  We also define the corresponding energy norm $\ndg{\cdot}$ for the
  \ipdg method by
  \begin{equation}\label{DG-norm}
    \ndg{v}
    :=\Big(\ltwo{\sqrt{\mat a}\dgrad v}{}^2
    +\ltwo{\sqrt{\sigma}\jump{v}}{\Ga}^2\Big)^{1/2},
  \end{equation}
  for $v\in \mathscript{S}$. Note that for $v\in S$, we have $\Pg
  \nabla v=\nabla v$ and, therefore, $B$ can be reduced to the more
  familiar form
  \begin{equation}
    \label{dg_elliptic_bilinear_fd}
    B(z,v)
    :=
    \int_{\W}(\mat a\dgrad z)\cdot\dgrad v
    +\int_{\Ga}\big(\theta\mean{\mat a\nabla v}\cdot\jump{z}
    -\mean{\mat a\nabla z}\cdot\jump{v}+\sigma\jump{z}\cdot\jump{v}\big),
  \end{equation}
  for $z,v\in S$ \cite[cf.]{arnold:82,arnold-brezzi-cockburn:02:unified,
    riviere-wheeler-girault:99,houston-schwab-suli:02}.
  Observe that both (\ref{cons}) and (\ref{norm_cons}) hold
  for the particular $B$ and $\ndg{\cdot}$ defined above.
\end{Definition}
\begin{Remark}[conforming part of a nonconforming finite element function]
  \label{rem:conforming-nonconforming}
  The space-discontinuous finite element space $S$ contains the
  conforming (continuous) finite element space $S_c=S\cap
  \sobhz1(\W)$ as a subspace. The approximation of functions in $S$
  by functions in $S_c$ will play an important role in our derivation
  of the \aposteriori bounds. This can be quantified in the following
  result, which is an extension of
  \cite[Thm. 2.1]{karakashian-pascal:07}.  For other similar results
  we refer to \cite{sun-wheeler:05, becker-hansbo-larson:03,
  houston-schoetzau-wihler:07, burman-ern:07}.
\end{Remark}
\begin{Lemma}[bounding the nonconforming part via jumps]
\label{h_oswald}
Suppose $\mathscript{T}$ is a regular mesh and $\mat a$ is
elementwise (weakly) differentiable. Then, for any function $Z\in S$
there exists a function $Z\conpart\in S_c$ such that
\begin{equation}\label{magic1}
\ltwo{Z-Z\conpart}{}\le C_1\ltwo{\sqrt{h}\jump{Z}}{\Ga},
\end{equation}
and
\begin{equation}\label{magic2}
\ltwo{\sqrt{\mat a}\dgrad(Z-Z\conpart)}{}\le C_2\ltwo{\sqrt{\sigma}\jump{Z}}{\Ga},
\end{equation}
where $C_1,C_2>0$ constants depending on the shape-regularity, on the maximum polynomial degree of the local basis and on $C_{\mat a}$.
\end{Lemma}
The proof, omitted here, follows closely that of
\cite[Thm. 2.2]{karakashian-pascal:03}.  Lemma \ref{h_oswald} can be
proved for irregular (i.e., with hanging-nodes) meshes
\cite[Thm. 2.3]{karakashian-pascal:03}, in which case $C_1$ and
$C_2$ depend on the maximum refinement and coarsening levels
$L_{\max}$.

\begin{Lemma}[\aposteriori bounds for \ipdg method for elliptic problem]
\label{ell_ipdg_apost}
Let $\mathscript{T}$ be a regular and $\mat a$ is elementwise (weakly)
differentiable.  Let $z$ and $Z$ be given by (\ref{elliptic_problem}) and
(\ref{dg_elliptic}).  Then
\begin{equation}
  \label{apost_ell_ipdg}
  \ndg{z-Z}\le  \estfunc(Z,\mat a,g,\mathscript{T}),
\end{equation}
where
\begin{equation}\label{est_ell}
\begin{aligned}
  \estfunc(Z,\mat a,g,\mathscript{T})
  := C\maxcondloc\bigg(&
  \ltwo{h/\sqrt{\amink}(g+\dgrad\cdot(\mat a\dgrad Z))}{}\\
  & + \ltwo{\sqrt{h/\amink}\jump{\mat a\dgrad Z}}{\Ga_{\dint}}
  +\ltwo{\sqrt{\sigma}\jump{Z}}{\Ga}\bigg),
\end{aligned}
\end{equation}
and $\maxcondloc:=\max_{\W}\sqrt{\amaxk/\amink}$, where $C>0$
depends only on $\mu(\cT)$ and $C_{\mat a}$.
\end{Lemma}
\begin{proof}
Our proof is inspired by \cite{karakashian-pascal:03,houston-schoetzau-wihler:07}.
Denoting by $Z\conpart\in S_c$ the conforming part of $W$ as
in Lemma \ref{h_oswald}, we have
\begin{equation}
  e:=z-Z=e_c+e_d,\quad\text{where}\quad e_c:=z-Z\conpart,
  \quad\text{and}\quad e_d:=Z\conpart-Z,
\end{equation}
yielding $e_c\in \sobhz1(\W)$. Thus, we have $B(z,e_c)=\langle
g,e_c\rangle$. Let $\Pg_0:\leb2(\W)\to\mathbb{R}$ denote the
orthogonal $\leb2$-projection onto the elementwise constant functions;
then $\Pg_0e_c \in S$ and we define $\eta:=e_c-\Pg_0e_c$.

We also have
\begin{equation}
  B(e,e_c)
  =B(z,e_c)-B(Z,e_c)
  =\langle g,e_c\rangle-B(Z,\eta)-B(Z,\Pg_0e_c)
  =\langle g,\eta\rangle-B(Z,\eta),
\end{equation}
which implies
\begin{equation}
  \label{ell_error_relation}
  \ltwo{\sqrt{a}\nabla e_c}{}^2
  =B(e_c,e_c)=\langle g,\eta\rangle-B(Z,\eta)-B(e_d,e_c).
\end{equation}
For the last term on the right-hand side of
(\ref{ell_error_relation}), we have
\begin{equation}
  \label{ell_before_first}
  |B(e_d,e_c)|\le \ltwo{\sqrt{\mat a}\dgrad e_d}{}\ltwo{\sqrt{\mat a}\nabla e_c}{}
  + \ha \sum_{s\subset\Ga}\sum_{\kappa=\kappa^+,\kappa^-} \amaxk|_{\kappa}\ltwo{\sqrt{h}(\Pg\nabla e_c)|_{\kappa}}{e}\ltwo{\jump{e_d}/\sqrt h}{s},
\end{equation}
where $\kappa^+$ and $\kappa^-$ are the (generic) elements having $e$
as common side. Using the inverse estimate of the form
$\ltwo{\sqrt{h}V}{e}\le C\ltwo{V}{\kappa}$ for $V=\Pg\nabla e_c$, and
the stability of the $\leb2$-projection, we arrive to
\begin{equation}
  |B(e_d,e_c)|\le \ltwo{\sqrt{\mat a}\dgrad e_d}{}\ltwo{\sqrt{\mat a}\nabla e_c}{}
  + C\maxcondloc\ltwo{\sqrt{\mat a}\nabla e_c}{}\ltwo{\sqrt{\sigma}\jump{e_d}}{\Ga}.
\end{equation}
Finally, noting that $\jump{e_d}=\jump{Z}$, and making use of (\ref{magic2}) we conclude that
\begin{equation}
  \label{ell_first}
  |B(e_d,e_c)|
  \le C\maxcondloc
  \ltwo{\sqrt{\mat a}\nabla e_c}{}\,\ltwo{\sqrt{\sigma}\jump{W}}{\Ga}.
\end{equation}

To bound the first two terms on the right-hand side of (\ref{ell_error_relation}), we begin by an elementwise integration by parts yielding
\begin{equation}\label{ell_sec}
\begin{aligned}
\langle g,\eta\rangle-B(Z,\eta)=& \int_{\W}\big(g+\dgrad\cdot(\mat a\dgrad Z)\big)\eta
-\int_{\Ga_{\dint}}\mean{\eta}\jump{\mat a\nabla Z}\ud s\\
&+\int_{\Ga}\theta\mean{\mat a\Pg\nabla \eta}\cdot\jump{Z}\ud s
-\int_{\Ga} \sigma\jump{Z}\cdot\jump{\eta}\ud s.
\end{aligned}
\end{equation}
The first term on the right-hand side of (\ref{ell_sec}) can be bounded as follows:
\begin{equation}\Big|
\int_{\W}\big(g+\nabla\cdot(\mat a\nabla Z)\big)\,\eta\,
\Big|
\le \ltwo{h/\sqrt{\amink}(g+\dgrad\cdot(\mat a\dgrad Z))}{}
\ltwo{\sqrt{\amink}h^{-1}\eta}{};
\end{equation}
upon observing that $\ltwo{h^{-1}\eta}{\kappa}\le C\ltwo{\nabla
e_c}{\kappa}$, this becomes
\begin{equation}
  \Big| \int_{\W}\big(g+\nabla\cdot(\mat a\nabla Z)\big)\,\eta \Big|
  \le C\maxcondloc
  \ltwo{h/\sqrt{\amink}(g+\dgrad\cdot(\mat a\dgrad Z))}{}\ltwo{\sqrt{\mat a}\nabla e_c}{}.
\end{equation}
\changes{
For the second term on the right-hand side of (\ref{ell_sec}), we use a trace estimate,
the bound $\ltwo{h^{-1}\eta}{\kappa}\le C\ltwo{\nabla
e_c}{\kappa}$ and we observe that $\nabla\eta=\nabla e_c$, to deduce}
\begin{equation}
\label{ell_third}
\Big|\int_{\Ga_{\dint}}\mean{\eta}\jump{\mat a\nabla Z}\ud s\Big|
\le
C\maxcondloc \ltwo{\sqrt{\mat a}\nabla e_c}{}\,\ltwo{\sqrt{h/\amink}\jump{\mat a\dgrad Z}}{\Ga_{\dint}}.
\end{equation}

\changes{For the third term on the right-hand side of (\ref{ell_sec}), we use
$\nabla\eta=\nabla e_c$ and, working alike to (\ref{ell_before_first}),}
we obtain
\begin{equation}
\label{ell_fourth}
\Big|\int_{\Ga}\theta\mean{\mat a\Pg\nabla \eta}\cdot\jump{Z}\Big|
\le
C\maxcondloc |\theta|\ltwo{\sqrt{\mat a}\nabla e_c}{}\,\ltwo{\sqrt{\sigma}\jump{Z}}{\Ga},
\end{equation}
and finally, for the last term on the right-hand side of (\ref{ell_sec}), we get
\begin{equation}
\label{ell_fifth}
\Big|\int_{\Ga}\sigma\jump{\eta}\cdot\jump{Z}\Big|
\le C\maxcondloc \ltwo{\sqrt{\mat a}\nabla e_c}{}\,\ltwo{\sqrt{\sigma}\jump{Z}}{\Ga}.
\end{equation}

The result follows combining the above relations.
\end{proof}
\begin{Theorem}[\aposteriori bounds for \ipdg method for parabolic problem]
\label{final_bound_semi}
Let $u, U$ be the exact weak solution of (\ref{pde}), and the \ipdg
solution of the problem (\ref{dg_semi}), respectively, and let $\mat
a$ be elementwise (weakly) differentiable.  Then, the following error
bound holds:
\begin{equation}\label{apost_semi_cor}
\begin{aligned}
\ltwo{u-U}{\leb2(0,T;\mathscript{S})}^2
\le& C\int_0^T\big(\estfunc^2(U,\mat a,AU-\Pg f+f,\mathscript{T})
+\amin^{-1}\ltwo{\sqrt{h}\jump{\partial_t U}}{\Ga}^2\big)\\
&+C\big(\ltwo{u_0-U(0)}{}+\ltwo{\sqrt{h}\jump{U(0)}}{}\big)^2.
\end{aligned}
\end{equation}
If we assume also that $u,U\in C(0,T;\sobhz1(\W))\cap \sobh1(0,T;\leb2(\W))$, then the following bound also holds:
\begin{equation}\label{apost_semi_l1_cor}
\begin{aligned}
\ltwo{u-U}{\leb2(0,T;\mathscript{S})}^2\le&
C\int_0^T\estfunc^2(U,\mat a,AU-\Pg f+f,\mathscript{T})
+C\Big(\int_0^T\ltwo{\sqrt{h}\jump{\partial_t U}}{\Ga}\Big)^2
\\
&+C\big(\ltwo{u_0-U(0)}{}+\ltwo{\sqrt{h}\jump{U(0)}}{}\big)^2.
\end{aligned}
\end{equation}
\end{Theorem}
\begin{proof} The results follow immediately from combining Theorems \ref{apost_semi_thm} and \ref{apost_semi_thm_l1} with
Lemma \ref{ell_ipdg_apost}, in conjunction with (\ref{magic1}).
\end{proof}

Finally, we give a result on useful properties of the \ipdg bilinear form and of the norm, which will be useful in \secref{sec:fully-discrete-bounds}.

\begin{Lemma}[continuity of $B$ and stability of the $\leb2$-projection]\label{auxiliary_results}
Consider the notation of \secref{sec:various-DG} and let $B$ and $\ndg{\cdot}$ denote
the \ipdg bilinear form (\ref{dg_elliptic_bilinear_fd})
and the DG-norm (\ref{DG-norm}). Then for $Z,V\in \fesh$ we have
\begin{equation}\label{continuity}
B(Z,V)\le C\maxcondloc\ndg{Z}\ndg{V}.
\end{equation}
Moreover, for $v\in \sobhz1(\W)$, the $\leb2$-projection is
DG-norm-stable, i.e.,
\begin{equation}\label{stability}
\ndg{\Pg v}\le C\maxcondloc\ndg{v}.
\end{equation}
\end{Lemma}
\begin{proof}
We omit the proof of (\ref{continuity}) which mimics that of (\ref{ell_first}).
For stability, note
\begin{equation}
\begin{aligned}
\ndg{\Pg v}^2=&\ltwo{\sqrt{\mat a}\dgrad(\Pg v-\Pg_0 v)}{}^2
+\ltwo{\sqrt{\sigma}\jump{v-\Pg v}}{\Ga}^2\\
\le&C\Big(\ltwo{\amaxk^{\ha} h^{-1}(\Pg v-\Pg_0 v)}{}^2
+\ltwo{\amaxk^{\ha} h^{-1}(v-\Pg v)}{}^2
+\ltwo{\amaxk^{\ha} \dgrad(v-\Pg v)}{}^2\Big)\\
\le&C\Big(\ltwo{\amaxk^{\ha} h^{-1}(v-\Pg_0 v)}{}^2
+\ltwo{\amaxk^{\ha} \dgrad(v-\Pg v)}{}^2\Big)\\
\le&C\ltwo{(\amaxk/\amink)^{\ha} \sqrt{\mat a}\nabla_{\mathscript{T}}v}{}^2,
\end{aligned}
\end{equation}
which implies (\ref{stability}).
\end{proof}

\section{\Aposteriori error bound for the fully discrete scheme}
\label{sec:fully-discrete-bounds}
In this section we discuss the abstract error analysis for the fully
discrete scheme defined in \S\ref{sec:fully-discrete-solution}.
\subsection{The elliptic reconstruction and the basic error relation}
Extend the sequence $\seq{\un}$ into a continuous piecewise linear
function of time:
\begin{equation}
  \label{eqn:fully-discrete-solution-extended}
  \U(0)=\UZERO\AND \U(t)=\elln(t)\un+\ellno(t)\uno,
\end{equation}
for $t\in\In$, and $\rangefromto n1N$,
where the functions $\elln$ and $\ellno$ are the Lagrange
basis functions
\begin{equation}
  \elln(t):=\frac{t-\tno}{\taun}\charfun{\clinter\tno\tn}
  +\frac{\tnp1-t}{\tau_{n+1}}\charfun{\clinter\tn{\tnp1}}.
\end{equation}
Using these time extensions and the (time and mesh dependent) discrete
elliptic operator $\discan$ of Definition \ref{elliptic_recon} with
respect to $\fes n$, defined by
\begin{equation}
  \discan Z\in\fes n\text{ such that }\ltwop{\discan Z}{V}
  =\bbilp n Z V\Forall V\in\fes n,
\end{equation}
we can write the scheme
(\ref{eqn:fully-discrete-Euler}) in the following pointwise form:
\begin{equation}
  \label{eqn:fully-discrete-Euler:pointwise}
  \pdt \U(t)+\discan\un=(\interpol n\uno-\uno)/\taun+\lproj n\fn,
\end{equation}
for all $t\in\opinter\tno\tn,\rangefromto n1N$.

We like to warn at this point that \emph{we use the same symbol $\U(t)$
to indicate the fully discrete solution time-extension in this
section, and the semidiscrete solution in \secref{sec:semidiscrete}.}
This should cause no confusion as long as the two cases are kept in
separate sections.

For each fixed $t\in\clinter0T$, and the corresponding $\rangefromto
n1N$ such that $t\in\opclinter\tno\tn$, we define the
\emph{time-dependent elliptic reconstruction} to be the function
$\w(t)\in\honezw$ satisfying
\begin{equation}
  \label{eqn:interpolated-reconstruction}
  \w(t)=\elln(t)\wn+\ellno(t)\wno_+,
\end{equation}
where $\wn\in\honezw$ is the \emph{elliptic reconstruction} of $\un$
defined implicitly as the (weak) solution of the elliptic problem with
data $\discan\un$, i.e., $\wn$ satisfies
\begin{gather}
  \operan\wn=\discan\un,
\end{gather}
and $\wno_+$ is the \emph{forward elliptic reconstruction} of $\uno$,
defined as the solution of the problem
\begin{equation}
  \opera^{n-1}\wno_+=\discano_+\interpol n\uno,
\end{equation}
where the operator $\funk{\discano_+}{\fes n}{\fes n}$ is defined by
\begin{equation}
\discano_+ Z\in\fes n\text{ such that }
  \ltwop{\discano_+Z}V
  =\bbilpsub{n-1}+ Z V\Forall V\in\fes n,
\end{equation}
$\bilform^{n-1}_+$ being the nonconforming bilinear
form corresponding to $\opera^{n-1}$, but with respect to the space
$\fes{n}$ (in contrast to $\bilform^{n-1}$ which is defined with
respect to $\fes{n-1}$).  For instance, for \ipdg, we have
\begin{multline}
  \bbilpsub{n-1}+Z V
  :=\sum_{\kappa\in\cT_n}\int_\kappa(\mat a(\tno)\grad Z)\inner\grad V
  \\
  +\int_{\Ga_n}\big(\theta\mean{\mat a(\tno)\Pg_n\grad V}\inner\jump Z
  -\mean{\mat a(\tno)\Pg_n\grad Z}\inner\jump V+\sigma_n\jump Z\inner\jump V\big).
\end{multline}
Using this definition of $\w$ on $\clinter\tno\tn$, the equation
(\ref{eqn:fully-discrete-Euler:pointwise}), implies
\begin{equation}
  \label{eqn:fully-discrete-Euler:pointwise-reconstructed}
  \pdt \U(t)+\opera(t)\w(t)
  =(\interpol n\uno-\uno)/\taun+\fn+\opera(t)\w(t)-\operan\wn.
\end{equation}
Subtracting the exact equation from this identity we obtain
\begin{equation}
  \label{eqn:discrete.parabolic-elliptic.relation.explicit}
  \pdt{\qb{\U-u}}+\opera\qb{\w-u}
  =
  \fraclpf{\interpol n\uno-\uno}\taun
  +\qp{\fn-f}
  +\qp{\opera\w-\operan\wn}
\end{equation}
for all $t\in\opinter\tno\tn$ and $\rangefromto n1N$, this leads to
the following technical basis of this section.
\begin{Lemma}[fully discrete error relation]
  \label{lem:discrete.parabolic-elliptic.relation}
  With the notation introduced in this section, let $e=\U-u$ (full
  error), $\rho:=\w-u$ (parabolic error) and $\epsi:=\w-\U$ (elliptic
  error).  Then we have
  \begin{equation}
    \label{eqn:discrete.parabolic-elliptic.relation}
    \pdt e+\opera\rho
    =
    \data n
    +
    \opera\w-\operan\wn,
  \end{equation}
  on $\clinter\tno\tn$, for all $\rangefromto n1N$.
\end{Lemma}
\begin{proof}
Replace the new notation for the errors in
(\ref{eqn:discrete.parabolic-elliptic.relation.explicit}).
\end{proof}
\subsection{Mesh interaction, DG spaces and decomposition}
\label{sec:mesh-interaction-decomposition}
The domain $\W$'s subdivisions \aka{meshes} $\seqs{\cT_n}{\rangefromto
  n0N}$ are \changes{
  assumed to be \emph{compatible} in the sense
  that for any two consecutive meshes, say $\cT_n$ and $\cT_{n-1}$, we
  have that $\cT_n$ is a constructed from $\cT_{n-1}$ in two main
  steps: (1) $\cT_{n-1}$ is locally coarsened by merging a chosen
  subset of elements then (2) the resulting coarsened mesh is locally
  refined\cite{lakkis-makridakis:06,LakkisPryer:10}.  This procedure
  leads to meshes which are \emph{locally} a refinement of one
  another.  For example, in the following diagram the mesh $\cT_{n-1}$
  has some elements marked (in red) for coarsening and (in blue) for
  refinement and the mesh $\cT_n$ can be thus obtained in the two steps:
}
\begin{equation}
  \begin{tikzpicture}
    \draw 
    (0,0)--(0,1) 
    node[pos=0.5,left]{$\cT_{n-1}$=}
    --(1,1)--(1,0)--cycle;
    \draw[line width=0pt,fill=blue!20] 
    (0.5,0.5)--++(0,0.5)--++(0.5,0)--++(0,-0.5)--cycle;
    \draw[line width=0pt,fill=red!20] 
    (0.5,0)--++(0,0.5)--++(0.5,0)--++(0,-0.5)--cycle;
    \draw 
    (.5,0)--++(0,1);
    \draw 
    (0,.5)--++(1,0);
    \draw 
    (.75,0)--++(0,.5);
    \draw 
    (0.5,.25)--++(.5,0);
    \draw[-stealth] 
    (1.5,.5)--++(1,0) 
    node[pos=0.5,above]{coarsen};
    \draw 
    (3,0)--(3,1)--(4,1)--(4,0)--cycle;
    \draw 
    (3.5,0)--++(0,1);
    \draw 
    (3,.5)--++(1,0);
    \draw[line width=0pt,fill=blue!20] 
    (3.5,0.5)--++(0,0.5)--++(0.5,0)--++(0,-0.5)--cycle;
    \draw[-stealth] 
    (4.5,.5)--++(1,0) node[pos=0.5,above]{refine};
    \draw 
    (6,0)--++(0,1)--++(1,0)-- 
    node[pos=0.5,right]{$=\cT_n$}++(0,-1)--cycle;
    \draw 
    (6.5,0)--++(0,1);
    \draw 
    (6,.5)--++(1,0);
    \draw 
    (6.75,0.5)--++(0,.5);
    \draw 
    (6.5,.75)--++(.5,0);
  \end{tikzpicture}
\end{equation}

For each $\rangefromto n1N$, we denote by $\check\cT_n$ the \emph{coarsest
common refinement} of $\cT_{n-1}$ and $\cT_n$.  The finite element
space corresponding to $\cT_n$ being $\fes n$, we shall be using the
space $\ccrfes n$ which is the finite element space with respect to
$\check\cT_n$.  Furthermore we denote by $\hopfes n:=\fes
n+\sobhz1(\W)$ and by $\hopfesh:=\sumifromto n0N\hopfes n$, the
minimal space that contains all these spaces.  These spaces are
equipped with the same type of norms given as in Section \ref{dg_heat}.

The conforming-nonconforming decomposition of $\U$, that we shall be
using is performed as follows:
\begin{LetterList}
\item
  For each given $\tn$, with $\rangefromto n1{N-1}$ we assume that
  the following two decompositions exist for $\un$,
  \begin{equation}
    \begin{aligned}
      &\un=\un\dispart+\un\conpart
      \text{ with respect to the mesh }
      \check\cT_n,
      \\
      \AND
      &\un=\un\dispartplus+\un\conpartplus
      \text{ with respect to the mesh }
      \check\cT_{n+1}.
    \end{aligned}
  \end{equation}
  Note that if the mesh changes, in general, $\un\dispart$ and
  $\un\dispartplus$ need not be equal functions.
\item
  For each $t\in\opclinter\tno\tn$, $\rangefromto n1N$, we define
  \begin{equation}
    \label{eqn:dis-con-part-time-extension}
    \U\dispart(t)
    :=\ellno(t)\uno\dispartplus+\elln(t)\un\dispart
    \AND
    \U\conpart(t)
    :=\ellno(t)\uno\conpartplus+\elln(t)\un\conpart.
  \end{equation}
\end{LetterList}

\begin{Definition}[\aposteriori error indicators]
  \label{def:fully-discrete:error-indicators}
  We set here some notation that is useful to state the main results
  concisely.  We make some assumptions in the process.  For each time
  interval $\clinter\tno\tn$, with $\rangefromto n1N$, we introduce
  \aposteriori error indicators as follows.
  \begin{LetterList}
  \item
    We assume that there exist $\C{els},\C{dgc}>0$ such that
    \begin{equation}
      \label{eqn:L2-projection-stability}
      \enorm{\lproj n v}\leq \C{els}\enorm{v},\Forall v\in\honezw,
    \end{equation}
    and
    \begin{equation}
      \label{eqn:bilinear-dgspace-continuity}
      \bbilp n Z V\leq\C{dgc}\enorm Z\enorm V ,\Forall Z,V\in\fes n.
    \end{equation}
    The \emph{time-stepping indicator} is given by
    \begin{equation}
      \label{eqn:fully-discrete:time-indicator}
      \indtime n
      :=\frac{\C{els}\C{dgc}}{\sqrt{3}}
      \enorm{\interpol n\uno-\un}.
    \end{equation}
  \item
    The \emph{time data-approximation indicator} is
    \begin{equation}
      \label{eqn:fully-discrete:time-data-indicator}
      \indftimeosc n:=
      \qp{\int_\tno^\tn\frac{\Norm{f(\tn)-f(s)}_{\sobh{-1}(\W)}^2}
    {\taun \amin(s)}\d s}^{1/2}
      .
    \end{equation}
  \item
    The \emph{mesh-change (or coarsening) indicator} is defined as
    \begin{equation}
      \label{eqn:fully-discrete:mesh-coarsening-indicator}
      \indcoarse n:=
      \frac{\Norm{\interpol n\uno-\uno}_\hmonew}{\taun}
      \qp{\frac1\taun\int_\tno^\tn\frac1\amin}^{1/2}.
    \end{equation}
  \item
    The \emph{parabolic nonconforming part indicator} is given
    by
    \begin{equation}
      \label{eqn:fully-discrete:nonconformity-indicator}
      \indnonconf n:=
      \frac{\Norm{\un\dispart-\uno\dispartplus}_{\sobh{-1}(\W)}}{\taun}
      \qp{\frac1\taun
    \int_\tno^\tn\frac1{\amin}}^{1/2},
    \end{equation}
    and the \emph{elliptic nonconforming part indicator} defined as
    \begin{equation}
      \label{eqn:fully-discrete:nonconformity-elliptic-indicator}
      \indnonconfel n:=
      \qp{\enorm{\un\dispart}^2+\enorm{\uno\dispartplus}^2}^{1/2}.
    \end{equation}
  \item
    The \emph{space (or elliptic) error indicator} is given by
    \begin{equation}
      \label{eqn:fully-discrete:elliptic-indicator}
      \indellipt n
      :=\cE(\un,\mat a(\tn),\discan\un,\cT_n),
    \end{equation}
    where $\cE$ is a particular choice of an energy-norm elliptic
    error estimator for the given spatial method.  Furthermore the
    \emph{forward elliptic error indicator}, due to mesh change, is
    given by
    \begin{equation}
      \label{eqn:fully-discrete:forward-elliptic-indicator}
      \indelliptplus n
      :=\cE(\interpol n\uno,\mat a(\tno),\discano_+\interpol n\uno,\cT_n).
    \end{equation}
  \item Consider first the auxiliary
    function of time
    \begin{equation}
      \difflipa n(s):=
      \Norm{\,\abs{\difflipA n s}_2}_{\leb\infty(\W)},
      \:s\in\clinter\tno\tn
    \end{equation}
    where the inner matrix norm is the Euclidean-induced one.  This
    definition is possible thanks to $\mat a$'s being symmetric positive
    definite. The function $\difflipa n$ is identically zero if the
    operator is time-independent, otherwise it acts like the numerator
    of $\mat a$'s normalized Hölder-continuity ratio.) Then we may
    define the following \emph{operator approximation indicators}
    \begin{gather}
      \label{eqn:fully-discrete:forward-operator-indicator}
      \indoperafor n
      :=
      \frac{\Norm{\discan\un}_{\sobh{-1}(\W)}}{\amin^n}
      \qp{\frac1\taun\int_\tno^\tn\elln^2\difflipa n^2}^{1/2}
      ,
      \\
      \label{eqn:fully-discrete:backward-operator-indicator}
      \indoperaback n
      :=
      \frac{\Norm{\discano_+\interpol n\uno}_{\sobh{-1}(\W)}}{\amin^{n-1}}
      \qp{\frac1\taun\int_\tno^\tn\ellno^2\difflipa{n-1}^2}^{1/2}
      ,
      \\
      \label{eqn:fully-discrete:mesh-operator-indicator}
      \indoperamesh n
      :=
      \Norm{\qb{\discano_+-\discan}\interpol n\uno}_{\sobh{-1}(\W)}
      \qp{\frac1\taun\int_\tno^\tn\frac{\ellno^2}{\amin}}^{1/2}
      ,
      \\
      \indopera n
      :=
      \indoperamesh n+\indoperafor n+\indoperaback n
      .
    \end{gather}
    \item
    \changes{
      Finally, the \emph{parabolic nonconforming part indicator of
        higher order} is given by
    \begin{equation}
      \label{eqn:fully-discrete:nonconformity-indicator-high}
      \kappa_n:=
      \frac{\Norm{\un\dispart-\uno\dispartplus}}{\taun}.
    \end{equation}}
  \end{LetterList}
\end{Definition}
\changes{
\begin{Remark}[computing the $\sobh{-1}(\W)$ norm]
  The norm $\sobh{-1}(\W)$ appearing in the indicators is easily
  computable at the cost of inverting a stiffness
  matrix~\cite{LakkisPryer:10}.  For many practical purposes, though
  this has to be replaced by the $\leb2(\W)$ norm times the
  Poincaré--Friedrichs constant $\C{PF}$ defined in
  (\ref{eqn:Poincare-Friedichs}) which implies the \emph{dual
    inequality}
  \begin{equation}
    \Norm{v}_{\sobh{-1}(\W)}\leq\C{PF}\Norm{v},\Forall v\in\leb2(\W).
   \end{equation}
  Note that this will not deteriorate most of the indicators.  The
  only indicators that may be affected by this change are
  $\indnonconf{n}$ and $\indnonconfel{n}$, and it may be possible to
  provide a sharp bound for negative Sobolev norms of $U_d$, but
  this seems to remain an open question at the time of writing.
\end{Remark}
}
\begin{Remark}[computing $\discan\un$ and similar terms]
  The operators $\discan$ and $\discan_+$ appearing in Definition
  \ref{def:fully-discrete:error-indicators}, can be realized in two
  ways in practice:
  \begin{LetterList}
  \item
    To save time, one can use the fully discrete scheme in pointwise
    form (\ref{eqn:fully-discrete-Euler:pointwise}) to evaluate some
    of these terms.  For example
    \begin{equation}
      \discan\un=\lproj n\fn-\fraclpf{\un-\interpol n\uno}\taun.
    \end{equation}
  \item
    The corresponding stiffness matrix could be computed and applied
    to the argument.  This seems to be necessary for $\discan_+$.
  \end{LetterList}
\end{Remark}
\begin{Remark}[mesh-change prediction]
  The mesh-change indicator $\indcoarse n$ can be \emph{precomputed}
  in a given computation.  Indeed, this term does not use explicitly
  any quantity deriving from the solution of the $n$-th Euler time-step
  (\ref{eqn:fully-discrete-Euler}).  This term is usually computable
  when a precise operator $\interpol n$ is available and it involves
  only local matrix-vector operations on each group of elements to be
  coarsened.
\end{Remark}
\begin{Remark}[an alternative time-stepping indicator]
  An equally valid definition for the time-stepping estimator
  $\indtime n$ can be given by
  \begin{equation}
    \indtime n:=
    \frac1{\sqrt3}
    \Norm{\discano_+\interpol n\uno-\discan\un}_{\sobh{-1}(\W)}.
  \end{equation}
  This alternative definition has the advantage of having no
  constants, but it is more complicated to compute and it must be
  reduced to the $\leb2(\W)$ norm by using the Poincaré--Friedrichs
  inequality.  A good side effect of this alternative choice is that
  in this case the indicator $\indoperamesh n$ vanishes; all other
  estimators remain unchanged.
\end{Remark}
\begin{Theorem}[abstract \aposteriori energy-error bound for Euler--DG]
  \label{lem:fully-discrete:fundamental-aposteriori-estimate}
  Let $\seqs{\un}n$ be the solution of
  (\ref{eqn:fully-discrete-Euler}) and $\U$ its time-extension as
  defined by (\ref{eqn:fully-discrete-solution-extended}) and $\w$ the
  elliptic reconstruction as defined by
  (\ref{eqn:interpolated-reconstruction}).  Then, with reference to
  Definition~\ref{def:fully-discrete:error-indicators}, for each $\rangefromto
  m1N$, we have
  \changes{
  \begin{equation}
  \begin{split}
    \label{eqn:fully-discrete:fundamental-aposteriori-estimate}
    \Norm{u-\U}_{\espace{0,\tm}}
    \leq &
      \Norm{u(0)-\U\conpart(0)}+3\parest m
    +{\sqrt2}\ellest m\\
    &+\qp{\frac12\sumifromto n1m\indnonconfel n^2\taun}^{1/2}
    +\sqrt{\frac{3}{2}}\sum_{n=1}^{m-1}\kappa_n\taun
    ,
    \end{split}
  \end{equation}
  }
  where the \emph{parabolic-error estimator} is defined as
  \begin{equation}
    \label{eqn:parabolic-error-esimator}
    \parest m
    :=\qp{\sumifromto n1m
    \qp{\indtime n+\indopera n+\indfosc n+\indcoarse n+\indnonconf n}^2\taun}^{1/2}
  \end{equation}
  and the \emph{elliptic estimator} is defined by
  \begin{equation}
    \ellest m:=\qp{
    \sumifromto n1m\qp{\indellipt n^2+\indelliptplus n^2}\taun}^{1/2}.
  \end{equation}
\end{Theorem}
We spread the proof in paragraphs
\ref{sec:energy-identity}--\ref{sec:fully-discrete:proof-basic-lemma}.
\subsection{The energy identity}
\label{sec:energy-identity}
As in the proof of Theorem \ref{apost_semi_thm} to get an energy
identity out of (\ref{eqn:discrete.parabolic-elliptic.relation}), we
will test with the error's \emph{conforming part}
\begin{equation}
  e\conpart:=e-\U\dispart=\rho+\epsic.
\end{equation}
Start with combining
(\ref{eqn:discrete.parabolic-elliptic.relation}) and definition
(\ref{eqn:interpolated-reconstruction}) to get
\begin{equation}
    \pdt e\conpart+\opera\rho
    =
    \pdt\U\dispart+\data n+\opera\w-\operan\wn
\end{equation}
Testing the above relation with $e\conpart$ we obtain the following
\emph{energy identity}:
\begin{equation}
  \label{eqn:discrete-energy-identity-differential}
  \begin{split}
    \frac12\dt\Norm{e\conpart}^2+\enorm\rho^2
    =&
    \ltwop{\pdt e\conpart}{e\conpart}+\bbil{\rho}{\rho}
    \\
    =&
    \bbil{\rho}{\epsic}
    +
    \ltwop{\pdt\U\dispart}{e\conpart}
    +
    \duality{\opera\w-\operan\wn}{e\conpart}
    \\&
    +
    \ltwop{\data n}{e\conpart}.
  \end{split}
\end{equation}
Integrating (\ref{eqn:discrete-energy-identity-differential}) from $0$
to $\tm\in\opclinter0T$, for an integer $m$, $1\leq m\leq N$ fixed, we
may write the integral form of the energy identity
\begin{equation}
  \label{eqn:discrete.linear.basic.integral.relation}
  \begin{split}
    \frac12\Norm{e\conpart(\tm)}^2+&\int_0^\tm\enorm\rho^2
    =
    \frac12\Norm{e\conpart(0)}^2+\int_0^\tm\bbil\rho\epsic
    \\
    &+\sumifromto n1m \Big(\int_{t_{n-1}}^\tn
    \duality{\opera\w-\operan\wn}{e\conpart}
    \\
    &+\int_{t_{n-1}}^\tn\ltwop{\data n}{e\conpart}
    \Big)
    \\
    &+\int_0^\tm\ltwop{\pdt\U\dispart}{e\conpart}
    +
    \changes{\frac12\sum_{n=1}^{m-1}\big(\Norm{u(\tn)-\un\conpartplus}^2-\Norm{e\conpart(\tn)}^2\big)}
    \\
    =&:\cI_0+\cI_1(\tm)+\cI_2(\tm)+\cI_3(\tm)+\cI_4(\tm)\changes{+\cI_5(\tm)}.
  \end{split}
\end{equation}
To obtain the \aposteriori error bound for scheme
(\ref{eqn:fully-discrete-Euler}), we now bound each of $\cI_i(\tm)$,
$\rangefromto i1{\changes5}$ ($\cI_0$ needs no bounding) appearing in relation
(\ref{eqn:discrete.linear.basic.integral.relation}), in terms of
either a-posteriori-computable or left-hand-side quantities.


A term that substantially distinguishes the fully discrete case from
the semidiscrete one discussed in \secref{sec:semidiscrete} is the
\emph{time-discretization} term $\cI_2(\tm)$, so we start by bounding
this term.
\subsection{Time discretization estimate}
To bound $\cI_2(\tm)$ we start by working out the first factor of the
integrand as follows
\begin{equation}
  \label{eqn:time-estimate-integrand-split}
  \begin{split}
    \opera(s)&\w(s)-\operan\wn
    =
    \opera(s)\qb{\elln(s)\wn+\ellno(s)\wno_+}-\operan\wn
    \\
    =&
    \elln(s)\qb{\opera(s)-\operan}\wn
    +\ellno(s)\qb{\opera(s)-\operano}\wno_+
    \\
    &+\ellno(s)\qp{\operano\wno_+-\operan\wn}
  \end{split}
\end{equation}
Since $\wn$ and $e\conpart$ are both in $\honezw$, we may bound the first term with
\begin{equation}
  \begin{split}
    \langle[\opera(s)&-\operan]\wn\,\vert\,{e\conpart(s)}\rangle
    =
    \int_\W\qp{\qp{\mat a(s)-\mat a(\tn)}\grad\wn}\inner\grad e\conpart(s)
    \\
    =&
    \int_\W\qp{
    \sqrt{\mat a(s)}
    \qp{\difflipA n s}
    \sqrt{\mat a(\tn)}\grad\wn}
    \inner\grad e\conpart(s)
    \\
    \leq
    &\Norm{\abs{\difflipA n s}_2}_{\leb\infty(\W)}
    \Norm{\sqrt{\mat a(s)}\grad e\conpart(s)}
    \Norm{\sqrt{\mat a(\tn)}\grad\wn}
    \\
    =&
    \difflipa n(s)
    \enorm{\wn}
    \enorm{e\conpart(s)}.
  \end{split}
\end{equation}
The second factor above can be bounded as follows
\begin{equation}
  \begin{gathered}
    \enorm{\wn}^2=\duality{\operan\wn}{\wn}
    =\ltwop{\discan\un}{\wn}
    =\Norm{\discan\un}_{\sobh{-1}(\W)}\Norm{\nabla\wn}
    \\
    \leq\frac{\Norm{\discan\un}_{\sobh{-1}(\W)}}{\amin^n}\enorm{\wn},
  \end{gathered}
\end{equation}
where the last step owes to the fact that
\begin{equation}
  \amin^n\Norm{\grad\wn}^2
  \leq\bbilp n\wn\wn
  =\enorm{\wn}^2.
\end{equation}
Thus $\enorm{\wn}\leq\fracl{\Norm{\discan\un}_{\sobh{-1}(\W)}}{\amin^n}$
and we obtain
\begin{equation}
  \begin{split}
    \int_\tno^\tn&\elln(s)
    \duality{\,\qb{\opera(s)-\operan}\wn}{e\conpart(s)}\d s
    \\
    &\leq
    \frac{\Norm{\discan\un}_{\sobh{-1}(\W)}}{\amin^n}
    \qp{\int_\tno^\tn\elln(s)^2
      \difflipa n(s)^2\d s
    }^{1/2}
    \qp{\int_\tno^\tn\enorm{e\conpart}^2}^{1/2}
    \\
    &=\indoperafor n\sqrt{\taun}\Norm{e\conpart}_{\leb2(\In;\hopfesh)}.
  \end{split}
\end{equation}
Similarly, we obtain
\begin{equation}
  \begin{split}
    \int_\tno^\tn&\ellno(s)
    \duality{\qb{\opera(s)-\operano}\wno_+}{e\conpart(s)}
    \d s
    \\
    &\leq
    \frac{\Norm{\discano_+\interpol n\uno}_{\sobh{-1}(\W)}}
     {\amin^{n-1}}
     \qp{\int_\tno^\tn\ellno(s)^2
       \difflipa{n-1}(s)^2\d s
     }^{1/2}
     \qp{\int_\tno^\tn\enorm{e\conpart}^2}^{1/2}
    \\
    &=\indoperaback n\sqrt{\taun}\Norm{e\conpart}_{\leb2(\In;\hopfesh)}.
  \end{split}
\end{equation}

To estimate the resultant of the integrand's third term in
(\ref{eqn:time-estimate-integrand-split}), we recall the elliptic
reconstruction's definition and note that in view of
(\ref{eqn:fully-discrete-Euler:pointwise}) we may write, for \changes{$n\geq1$},
that
\begin{multline}
  \label{eqn:discrete.time-estimate:proof:rec-to-comp}
    \duality{\operano\wno_+-\operan\wn}{e\conpart}
    =
    \ltwop{\discano_+\interpol n\uno-\discan\un}{e\conpart}
    \\
    =
    \ltwop{\qb{\discano_+-\discan}\interpol n\uno}{e\conpart}
    +
    \ltwop{\discan\qb{\interpol n\uno-\un}}{\lproj n e\conpart}
\end{multline}
given that $\discan\interpol n\uno,\discan\un\in\fes n$.  The first
term on the right-hand side of
(\ref{eqn:discrete.time-estimate:proof:rec-to-comp}) is simply bounded by
\begin{equation}
  \ltwop{\qb{\discano_+-\discan}\interpol n\uno}{e\conpart}
  \leq
  \Norm{\qb{\discano_+-\discan}\interpol n\uno}_{\sobh{-1}(\W)}
  \Norm{\grad e\conpart},
\end{equation}
and thus, recalling definition
(\ref{eqn:fully-discrete:mesh-operator-indicator}), we have
\begin{equation}
  \int_\tno^\tn\ellno\ltwop{\qb{\discano_+-\discan}\interpol n\uno}{e\conpart}
  \leq
  \indoperamesh n\sqrt{\taun}\Norm{e\conpart}_{\leb2(\In;\hopfesh)}.
\end{equation}

The second term on the right-hand side of
(\ref{eqn:discrete.time-estimate:proof:rec-to-comp}) can be given a
simpler expression as follows:
\begin{equation}
  \begin{split}
    \ltwop{\discan\qb{\interpol n\uno-\un}}{\lproj n e\conpart}
    =&
    \bbilp n{\interpol n\uno-\un}{\lproj n e\conpart}
    \\
    \leq&
    \C{dgc}\enorm{\interpol n\uno-\un}\enorm{\lproj n e\conpart}
    \\
    \leq&
    \C{dgc}\C{els}\enorm{\interpol n\uno-\un}\enorm{e\conpart}
  \end{split}
\end{equation}
thanks to the stability of $\lproj n$ with respect to the energy norm
$\enorm\cdot$ assumed in (\ref{eqn:L2-projection-stability}).
Therefore, recalling definition
(\ref{eqn:fully-discrete:time-indicator}), we obtain
\begin{equation}
    \int_\tno^\tn
    \ellno\ltwop{\discan\qb{\interpol n\uno-\un}}{\lproj n e\conpart}
    \leq\indtime n\sqrt{\taun}\Norm{e\conpart}_{\leb2(\In;\hopfesh)}
\end{equation}

The time error estimate follows:
\begin{equation}
  \label{eqn:fully-discrete-estimate:time-indicator}
  \cI_2(\tm)
  \leq
  \sumifromto n1m\qp{\indopera n+\indtime n}\sqrt{\taun}
  \qp{\Norm\rho_{\espace{\tno,\tn}}
    +\Norm\epsic_{\espace{\tno,\tn}}
  }.
\end{equation}
\subsection{The other error estimates}
To bound the spatial error term, $\cI_1(t)$
in (\ref{eqn:discrete.linear.basic.integral.relation}), we simply
consider
\begin{equation}
  \label{eqn:fully-discrete-estimate:spatial-indicator}
  \cI_1(\tm)=
  \int_0^\tm\bbil\rho\epsic
  \leq
  \int_0^\tn\enorm{\rho}\enorm{\epsic},
\end{equation}
with the aim of absorbing the first factor in the
left-hand side of (\ref{eqn:discrete.linear.basic.integral.relation})
and using an elliptic error estimator to bound the second term.
\par
The term $\cI_3(\tm)$ in
(\ref{eqn:discrete.linear.basic.integral.relation}) which takes into
account data approximation:
\begin{equation}
  \cI_3(\tm)
  =
  \sumifromto n1m\int_\tno^\tn
  \ltwop{\frac{\interpol n\uno-\uno}\taun+\fn-f}{e\conpart}
  .
\end{equation}
The first term can be bounded by using the
$\pairing{\sobh{-1}(\W)}{\sobh1_0(\W)}$ pairing, as we did with the
time-estimator above:
\begin{multline}
  \sumifromto n1m\int_\tno^\tn
  \ltwop{\fraclpf{\interpol n\uno-\uno}\taun+\fn-f}{e\conpart}
  \\
  \leq
  \sumifromto n1m
  \qp{\indcoarse n+\indftimeosc n}\sqrt{\taun}
  \Norm{e\conpart}_{\espace{\tno,\tn}}
  ,
\end{multline}
where we have used definitions
(\ref{eqn:fully-discrete:time-data-indicator})
and~(\ref{eqn:fully-discrete:mesh-coarsening-indicator}).

Hence we obtain the bound
\begin{equation}
  \label{eqn:fully-discrete-estimate:data-indicator}
  \cI_3(\tm)
  \leq
  \sumifromto n1m\qp{\indfosc n+\indcoarse n}\sqrt{\taun}
  \qp{\Norm\rho_{\espace{\tno,\tn}}
    +\Norm\epsic_{\espace{\tno,\tn}}}.
\end{equation}
\par \changes{We estimate the second-last term} on the right-hand side of
(\ref{eqn:discrete.linear.basic.integral.relation}).  This term can be
bounded in two different ways.  For concision's sake we expose only the
estimate that yields smaller accumulation over long integration-times:
\begin{equation}
  \label{eqn:fully-discrete-estimate:nonconforming-indicator}
  \begin{split}
    \cI_4(\tm)
    =&\int_0^\tm\ltwop{\pdt \U\dispart}{e\conpart}
    \leq
    \sumifromto n1m
    \int_\tno^\tn
      \Norm{\pdt \U\dispart}_{\sobh{-1}(\W)}
      \Norm{\grad e\conpart}
      \\
    \leq&
    \sumifromto n1m
    \indnonconf n\sqrt{\taun}
    \qp{\Norm\rho_{\espace{\tno,\tn}}
      +\Norm{\epsic}_{\espace{\tno,\tn}}},
  \end{split}
\end{equation}
by recalling (\ref{eqn:fully-discrete:nonconformity-indicator}).
\par 
\changes{Observing the identity
  \begin{equation}
    \Norm{u(\tn)-\un\conpartplus}^2-\Norm{e\conpart(\tn)}^2
    =
    \Norm{\un\dispart-\un\dispartplus}^2
    +
    \ltwop{\un\dispart-\un\dispartplus}{e\conpart(\tn)},
  \end{equation}
we estimate the last term on the right-hand side of
(\ref{eqn:discrete.linear.basic.integral.relation}), as follows:
\begin{equation}\label{eqn:discrete.jump.indicator}
\begin{split}
    \cI_5(\tm)
    = &
\frac12\sum_{n=1}^{m-1}\big(\Norm{\un\dispart-\un\dispartplus}^2
+
\ltwop{\un\dispart-\un\dispartplus}{e\conpart(\tn)}\big)\\
\le &
\frac12\sum_{n=1}^{m-1}\big(\kappa_n^2\taun^2+\max_{1\le l\le m-1}\Norm{e\conpart(t_l)}\kappa_n\taun\big)\\
\le &
\frac{3}{4}\Big(\sum_{n=1}^{m-1}\kappa_n\taun\Big)^2
+\frac{1}{4}\max_{1\le l\le m-1}\Norm{e\conpart(t_l)}^2
\end{split}
\end{equation}}
\subsection{
  Concluding the proof of Theorem
  \ref{lem:fully-discrete:fundamental-aposteriori-estimate} }
\label{sec:fully-discrete:proof-basic-lemma}
Combining the energy relation
(\ref{eqn:discrete.linear.basic.integral.relation}) with the bounds
(\ref{eqn:fully-discrete-estimate:time-indicator}),
(\ref{eqn:fully-discrete-estimate:spatial-indicator}),
\changes{
(\ref{eqn:fully-discrete-estimate:data-indicator}),
(\ref{eqn:fully-discrete-estimate:nonconforming-indicator})
and
(\ref{eqn:discrete.jump.indicator}), we obtain}
\changes{
\begin{equation}\label{eqn:error.one}
  \begin{split}
  &      \frac12\Norm{e\conpart(\tm)}^2
     +
    \Norm{\rho}_{\espace{\tno,\tn}}^2\\
    \leq &
    \frac12\Norm{e\conpart(0)}^2+\frac{3}{4}\Big(\sum_{n=1}^{m-1}\kappa_n\taun\Big)^2
+\frac{1}{4}\max_{1\le l\le m-1}\Norm{e\conpart(t_l)}^2
    \\
    &+
    \sumifromto n1m
    \qp{\indtime n+\indopera n+\indfosc n+\indcoarse n+\indnonconf n}
    \sqrt{\taun}
    \Norm\epsic_{\espace{\tno,\tn}}
    \\
    &+
    \sumifromto n1m
    \qp{
      \qp{\indtime n+\indopera n+\indfosc n+\indcoarse n+\indnonconf n}
      \sqrt{\taun}
      +\Norm\epsic_{\espace{\tno,\tn}}}
    \Norm\rho_{\espace{0,\tm}}
    .
  \end{split}
\end{equation}
}
\changes{
Choosing $m=m_*$ so that $\Norm{e\conpart(t_{m_*})}=\max_{1\le l\le m-1}\Norm{e\conpart(t_l)}$
in (\ref{eqn:error.one}), yields a bound on $\max_{1\le l\le m-1}\Norm{e\conpart(t_l)}^2/4$, which is then used
again to bound the third term on the right-hand of (\ref{eqn:error.one}), resulting to
\begin{equation}\label{eqn:error.two}
  \begin{split}
  \Norm{\rho}&_{\espace{\tno,\tn}}^2
    \leq
    \Norm{e\conpart(0)}^2+\frac{3}{2}\Big(\sum_{n=1}^{m-1}\kappa_n\taun\Big)^2
    \\
    &+
    2\sumifromto n1m
    \qp{\indtime n+\indopera n+\indfosc n+\indcoarse n+\indnonconf n}
    \sqrt{\taun}
    \Norm\epsic_{\espace{\tno,\tn}}
    \\
    &+
    2\sumifromto n1m
    \qp{
      \qp{\indtime n+\indopera n+\indfosc n+\indcoarse n+\indnonconf n}
      \sqrt{\taun}
      +\Norm\epsic_{\espace{\tno,\tn}}}
    \Norm\rho_{\espace{0,\tm}}
    ,
  \end{split}
\end{equation}}
which is an inequality of the form
\begin{equation}
  \abs{\vec a}^2
  \leq c^2+\vec d\inner{\vec b}+\qp{\vec d+\vec b}\inner{\vec a},
\end{equation}
where $\vec a,\vec b,\vec d\in\R{m+1}$ and $c\in\reals$ are
appropriately chosen.  It follows that
\begin{equation}
  \abs{\vec a}\leq
  \max\setof{\abs c,\abs{\vec d}}
  +\abs{\vec d}+\abs{\vec b},
\end{equation}
which, using the notation introduced in the statement of the theorem,
implies
\begin{equation}
  \label{eqn:fully-discrete:proof-of-main-theorem:pre-elliptic}
  \Norm\rho_{\espace{0,\tm}}\leq
  \Norm{e\conpart(0)}
  +
  \sqrt{\frac{3}{2}}\sum_{n=1}^{m-1}\kappa_n\taun
  +
  3\parest m
  +
  \Norm{\epsic}_{\espace{0,\tm}}.
\end{equation}
To close the estimate, the last term on the right-hand side of
(\ref{eqn:fully-discrete:proof-of-main-theorem:pre-elliptic}) is
bounded by
\begin{equation}
  \Norm{\epsic}_{\espace{0,\tm}}
  \leq
  \Norm{\epsi}_{\espace{0,\tm}}
  +
  \Norm{U\dispart}_{\espace{0,\tm}}.
\end{equation}
The first term yields
\begin{equation}
  \begin{split}
    \Norm{\epsi}_{\espace{0,\tm}}^2
    =&
    \sumifromto n1m\int_\tno^\tn\enorm{\ellno(\wn_+-\un)+\elln\epsic^n}^2
    \\
    \leq&
    \sumifromto n1m\int_\tno^\tn\ellno\indelliptplus n^2+\elln\indellipt n^2
    \\
    \leq&
    \frac12\sumifromto n1m\qp{\indelliptplus n^2+\indellipt n^2}\taun
    =\frac12\ellest m^2.
  \end{split}
\end{equation}
Similarly, the second term yields
\begin{equation}
  \Norm{U\dispart}_{\espace{0,\tm}}^2
  \leq
  \frac12\sumifromto n1m\indnonconfel n^2\taun.
\end{equation}
Merging these inequalities with
(\ref{eqn:fully-discrete:proof-of-main-theorem:pre-elliptic}) and
using the triangle inequality we obtain
\begin{equation}
  \Norm e_{\espace{0,\tm}}
  \leq
  \Norm\epsi_{\espace{0,\tm}}+\Norm\rho_{\espace{0,\tm}}
\end{equation}
we obtain
(\ref{eqn:fully-discrete:fundamental-aposteriori-estimate}).
\begin{Remark}[short-time integration]
  In the spirit of Theorem \ref{apost_semi_thm_l1}, it is possible to modify Theorem
  \ref{lem:fully-discrete:fundamental-aposteriori-estimate} and the
  appropriate indicators as to accommodate a short time-integration
  version of this result where $\leb1$-accumulation in time replaces
  the $\leb2$-accumulation for certain estimators.  Over shorter
  time-intervals this provides a tighter bound.
\end{Remark}
\begin{Theorem}[\aposteriori energy-error bound for Euler--\ipdg]
\label{energy_error_bound_for_Euler_ipdg_thm}
  Under the same assumptions of Theorem
  \ref{lem:fully-discrete:fundamental-aposteriori-estimate}, assuming we employ the \ipdg method in space as described in
  \secref{sec:various-DG}, the error bound
  (\ref{eqn:fully-discrete:fundamental-aposteriori-estimate}) holds
  with the estimators $\parest m$, $\ellest m$ and $\sumifromto
  n1m\indnonconfel n^2\taun$ explicitly computable as follows:
  \begin{LetterList}
  \item
    for the elliptic indicators $\indellipt n$ and
    $\indelliptplus{n}$, replace $\cE$ by $\estfunc$, as defined in
    (\ref{est_ell}), into
    (\ref{eqn:fully-discrete:elliptic-indicator}) and
    (\ref{eqn:fully-discrete:forward-elliptic-indicator}),
    respectively;
  \item
    for the nonconforming part indicators $\indnonconf n$ and
    $\indnonconfel n$, respectively, (cf. \secref{sec:mesh-interaction-decomposition} and Lemma \ref{h_oswald}),
    we replace
    \begin{equation}
      \Norm{\un\dispart-\uno\dispartplus}_{\sobh{-1}(\W)},
      \quad
      \enorm{\un\dispart}
      \AND
      \enorm{\uno\dispartplus},
    \end{equation}
    respectively, by
    \begin{equation}
      \C{PF}\C1\Norm{\smash{\sqrt{\check h}_n}\jump{\un-\uno}},
      \quad
      \C2\Norm{\sqrt{\check\sigma_n}\jump{\un}},
    \end{equation}
    where $\check h_n$ is the mesh-size function of $\check\cT_n$ and
    $\check\sigma_n$ is related to it via
    (\ref{eqn:def:penalty-sigma});
  \item
    replace all $\sobh{-1}(\W)$ norms by $\C{PF}$ times the
    $\leb2(\W)$ norm.
  \end{LetterList}
\end{Theorem}
\changes{
\section{Computer experiments} 
\label{sec:computer-experiments}
In this final section we summarize the results of computer experiments
aimed at testing the efficiency and reliability of the fully discrete
estimators derived in \S~\ref{sec:fully-discrete-bounds}.  We built
our code upon the free finite element software
\Program{FEniCS}~\cite{Logg:fenics} while \matlab was used as an
end-tool to visualize the time-behavior of various estimators.

All the computational examples are in space dimension $d=2$ and their
choice is such as to illustrate as many aspect as possible of
practical convergence rate \aka{experimental order of convergence, in
  short EOC} and the effectivity index (EI), on uniform space-time
meshes, of the proposed a posteriori error indicators defined in
\S~\ref{def:fully-discrete:error-indicators}.
}
\changes{
\subsection{Benchmark solutions}
We consider three benchmark problems for which $u_0$ and $f$ are
chosen so that the exact solution $u$ of problem (\ref{pde}) coincides
with one of the following \emph{benchmark solutions}:
\begin{gather}
   \label{eqn:numerics:example1} 
   u_1(x,y,t) 
   = \sin(\pi t)\sin^2(\pi x)\sin^2(\pi y)
   ,
   \\
   \label{eqn:numerics:example2}
   u_2(x,y,t) 
   = \hat u_2(r,\phi,t) 
   = \sin(\pi t)(r^2 \cos^2(\phi) - 1)^2 (r^2 \sin^2(\phi) - 1)^2 r^{z_0} g(\phi) 
   ,\\
  \label{eqn:numerics:example3}
   u_3(x,y,t) 
   = \sin(20 \pi t)\sin^2(\pi x)\sin^2(\pi y)
\end{gather}
for $t\in [0,1]$ and
\begin{equation}
  (x,y)\in 
  \begin{cases}
    (0,1) \times (0,1)
    &\text{ in (\ref{eqn:numerics:example1}) and (\ref{eqn:numerics:example3})} 
    \\ 
    (-1,1)^2\setminus [0,1)\times
      (-1,0]
    &
    \text{ in (\ref{eqn:numerics:example2})}.
  \end{cases}
\end{equation}

To complete the definition of $u_2$ in (\ref{eqn:numerics:example2}) we
consider
\begin{equation}
  z_0:=0.544483736782464\text{ such that }
  \sin^2(z_0 \omega) = z_0^2 \sin^2(\omega),
  \text{ with }\omega=\fracl{3\pi}{2},
\end{equation}
and 
\begin{equation}
\begin{aligned}
  g(\phi)
  &
  := 
  \qpbigg{ 
    \frac{1}{z_0-1}\sin((z_0-1)\omega) 
    - \frac{1}{z_0+1}\sin((z_0+1)\omega)
  }
  \\
  &
  \phantom{:=}
  \times(\cos((z_0-1)\phi) - \cos((z_0+1)\phi)) 
  \\
  &
  \phantom{:=}
  -\qpbigg{
    \frac{1}{z_0-1}\sin((z_0-1)\phi) 
    - \frac{1}{z_0+1}\sin((z_0+1)\phi) 
  }
  \\
  &
  \phantom{:=-}
  \times(\cos((z_0-1)\omega) 
  - \cos((z_0+1)\omega)).
\end{aligned}
\end{equation}
It is well-known~\cite{grisvard:92,brenner-gudi-sung:09} that the
gradient of $u_2$ in (\ref{eqn:numerics:example2}) has a singularity
at the reentrant corner located at the origin of $\W$.

Solution $u_1$ is smooth and varies ``slowly'' in time.  Solution
$u_3$ is also smooth by it oscillates much faster and is used to
emphasize the time-error indicator appearing in the parabolic error
estimator $\parest m$, defined
in~(\ref{eqn:parabolic-error-esimator}).

Similar examples have been studied elsewhere, for example in
\cite{lakkis-makridakis:06,LakkisPryer:10}.

Note that the diffusion tensor, $\mat a(\vec x,t)$, is a constant
function (equal to $1$) of space-time and that the initial error
$\enorm{u(0)-U(0)}=0$ in all examples.
}
\changes{
\subsection{Computed quantities}
In each of the examples, we compute the solution of
(\ref{eqn:fully-discrete-Euler}) using finite element spaces
consisting of polynomials of degree $p$ equal to $1,2$ and $3$ with
interior penalty parameter $C_{\mat a,\mu(\cT)}$ in
(\ref{eqn:def:penalty-sigma}) having values $40, 80$ and $160$
respectively which are sufficient to guarantee stability of the
numerical scheme.

We study the asymptotic behavior of the indicators by setting all
constants appearing in Theorem
\ref{energy_error_bound_for_Euler_ipdg_thm} equal to $1$ and
monitoring the evolution of the values and experimental order of
convergence of the estimators and the error as well as effectivity
index over time on a sequence of uniformly refined meshes with a fixed
time step $\tau$ and polynomial degree $p$. For this purpose, we
define \emph{experimental order of convergence}, in symbols $\EOC$, of
a given sequence of positive quantities $a(i)$ defined on a sequence
of meshes of size $h(i)$ by
\begin{equation}
  \EOC( a,i ) = \frac{\log(a(i+1)/a(i))}{\log(h(i+1)/h(i))}
\end{equation}
and the \emph{inverse effectivity index}, $\EI$, by 
\begin{equation}
  \EI = \frac{\Norm e_{L_2(0,t_m;\mathcal{T})}}{\parest m + \ellest m}.
\end{equation}
We use the \emph{inverse effectivity index}, instead of the (direct)
effectivity index, because it is easier to visualize while conveying
the same information.  It also has the advantage of relating directly
to the constants appearing in Theorem
\ref{energy_error_bound_for_Euler_ipdg_thm}.
}
\subsection{Conclusions}
\changes{
The numerical experiments clearly indicate that the error estimators
are reliable (as expected from the theory) and efficient.  This is
clearly seen by the match in EOC between the error and the two main
estimators $\ellest m$ and $\parest m$ for each $m$.

Since we use time-invariant finite element spaces, the mesh-change
estimators are null and do not influence the estimators.

The nonconforming indicator $\pregsqrt{\frac12\sumifromto
  n1m\indnonconfel n^2\taun}$ was found to be of higher order with
respect the elliptic estimator, $\ellest m$.  This is most likely to
be an effect of using time-invariant meshes and the nonconforming
indicator can be safely ignored as long as the mesh does not change.

Adding mesh change, space-time-dependent diffusion $\mat a$, and
variable time-step to our numerical experiments will exhibit more
properties of the estimators but we eschew deeper numerical
experiments in this paper for concision's sake.  For the same reason,
the derivation of adaptive methods based on our indicators is omitted
here.

Results for problem (\ref{eqn:numerics:example1}) with $p=1$, problem
(\ref{eqn:numerics:example2}) with $p=2$ and problem
(\ref{eqn:numerics:example3}) with $p=3$ are depicted and commented
further in figures \ref{fig:e1p1}, \ref{fig:e2p2} and \ref{fig:e3p3}
respectively.
}
\newcommand{\figscale}{0.50}
\newcommand{\jpgscale}{0.35}
\newcommand{\figwidth}{0.95\textwidth}
\begin{figure}[ht]
  \caption{\label{fig:e1p1}Example with exact solution $u_1$, given by
    (\ref{eqn:numerics:example1}), approximated with piecewise
    polynomials of degree $p=1$.}
  \begin{center}
    \subfigure[{
        Mesh-size $h(i)=2^{-i/2}$, $i=2,\dotsc,8$, and timestep
        $\tau=0.1\,h$. On top we plot the EOC of the single
        cumulative indicators $\parest m$ and $\ellest m$. Both
        indicators have the same asymptotic $\EOC\approx1$ as has the
        error. The effectivity index tends towards $1/0.12$.
    }]{
      \includegraphics[trim =25pt 250pt 25pt 245pt, clip=true,width=\figwidth]{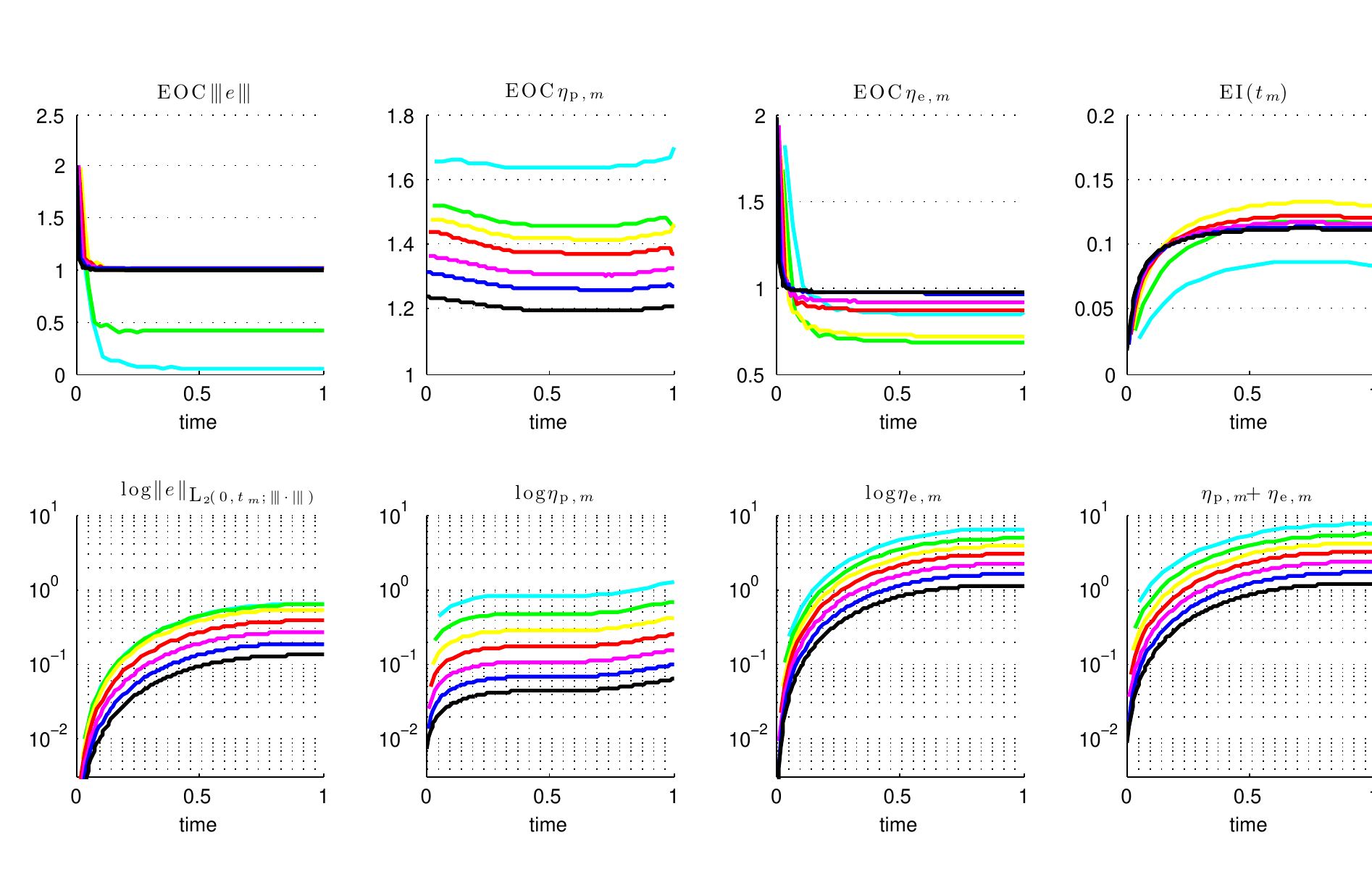}
    }
    \subfigure[{
        Mesh-size $h(i)=2^{-i/2}$, $i=2,\dotsc,8$, and timestep
        $\tau=0.1\,h^2$. On top we plot the EOC of the single
        cumulative indicators $\parest m$ and $\ellest m$. Both
        indicators have the same asymptotic $\EOC\approx1$ as has the
        error. The effectivity index tends towards $1/0.12$.
    }]{
      \includegraphics[trim =25pt 250pt 25pt 245pt, clip=true,width=\figwidth]{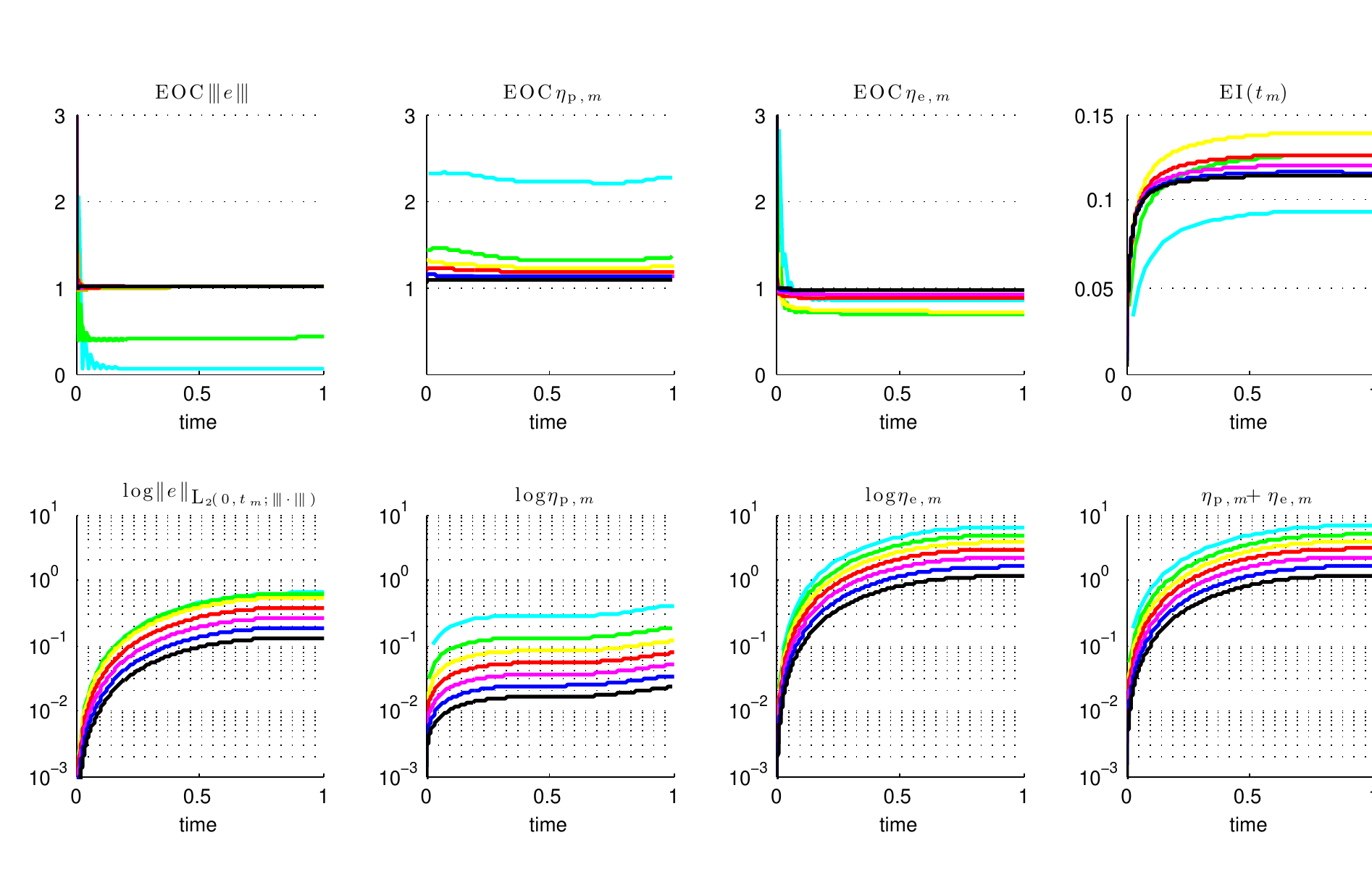}
    }
  \end{center}
\end{figure}
\begin{figure}[ht]
  \caption{\label{fig:e2p2}Example (\ref{eqn:numerics:example2}) with
    piecewise polynomials of degree $p=2$.}
  \begin{center}
    \subfigure[{
        \label{fig:} Mesh-size $h(i)=2^{-i/2}$, $i=2,\dotsc,5$, and
        timestep $\tau=0.1\,h^2$. On top we plot the EOC of the
        single cumulative indicators $\parest m$ and $\ellest
        m$. $EOC < 1$ for the error is due to lack of
        $H^2$-regularity. Note that the elliptic estimator has
        asymptotically the same EOC as the error.  
    }]{
      \includegraphics[trim =25pt 250pt 25pt 245pt, clip=true,width=\figwidth]{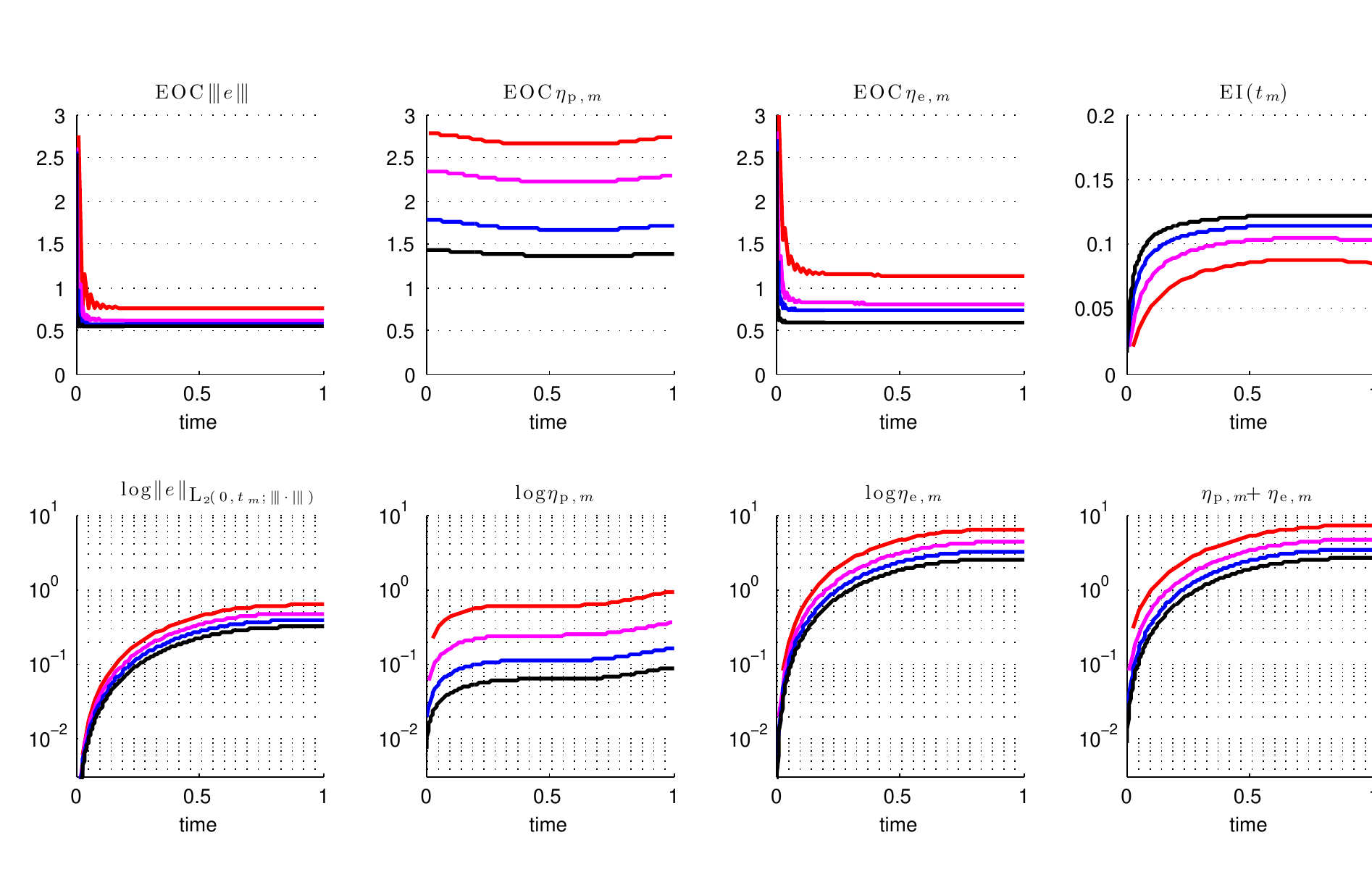}
    }
    \subfigure[{
        Mesh-size $h(i)=2^{-i/2}$, $i=2,\dotsc,4$ and timestep $\tau =
        0.1 h^3$. On top we plot the EOC of the single cumulative
        indicators $\parest m$ and $\ellest m$. $EOC<1$ for the
        error is due to lack of $H^2$-regularity. Note that the
        elliptic estimator has asymptotically the same EOC as the
        error.
    }]{
      \includegraphics[trim =25pt 250pt 25pt 245pt, clip=true,width=\figwidth]{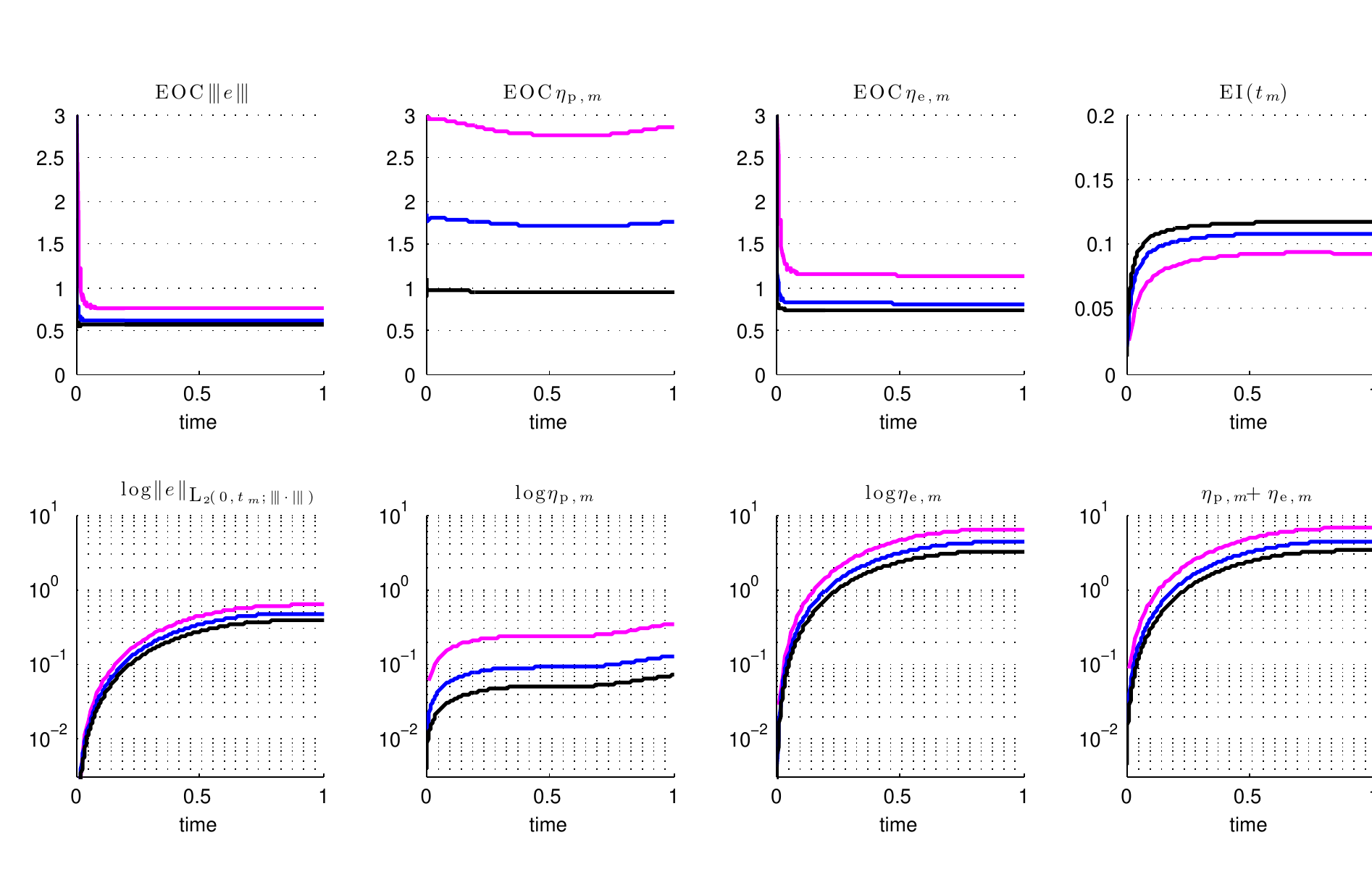}
    }
  \end{center}
\end{figure}
\begin{figure}[ht]
  \caption{\label{fig:e3p3} Example (\ref{eqn:numerics:example3}) with
    discontinuous piecewise polynomials of degree $p=3$.}
  \begin{center}
    \subfigure[{
        Mesh-size $h(i)=2^{-i/2}$, $i=2,\dotsc,5$ and timestep
        $\tau=0.1\,h^3$. On top we plot the EOC of the single
        cumulative indicators $\parest m$ and $\ellest m$. Both
        indicators have the same asymptotic $EOC \approx 1$ as has the
        error. The effectivity index tends towards asymptotic value
        $200$.
    }]{
      \includegraphics[trim =25pt 250pt 25pt 245pt, clip=true,width=\figwidth]{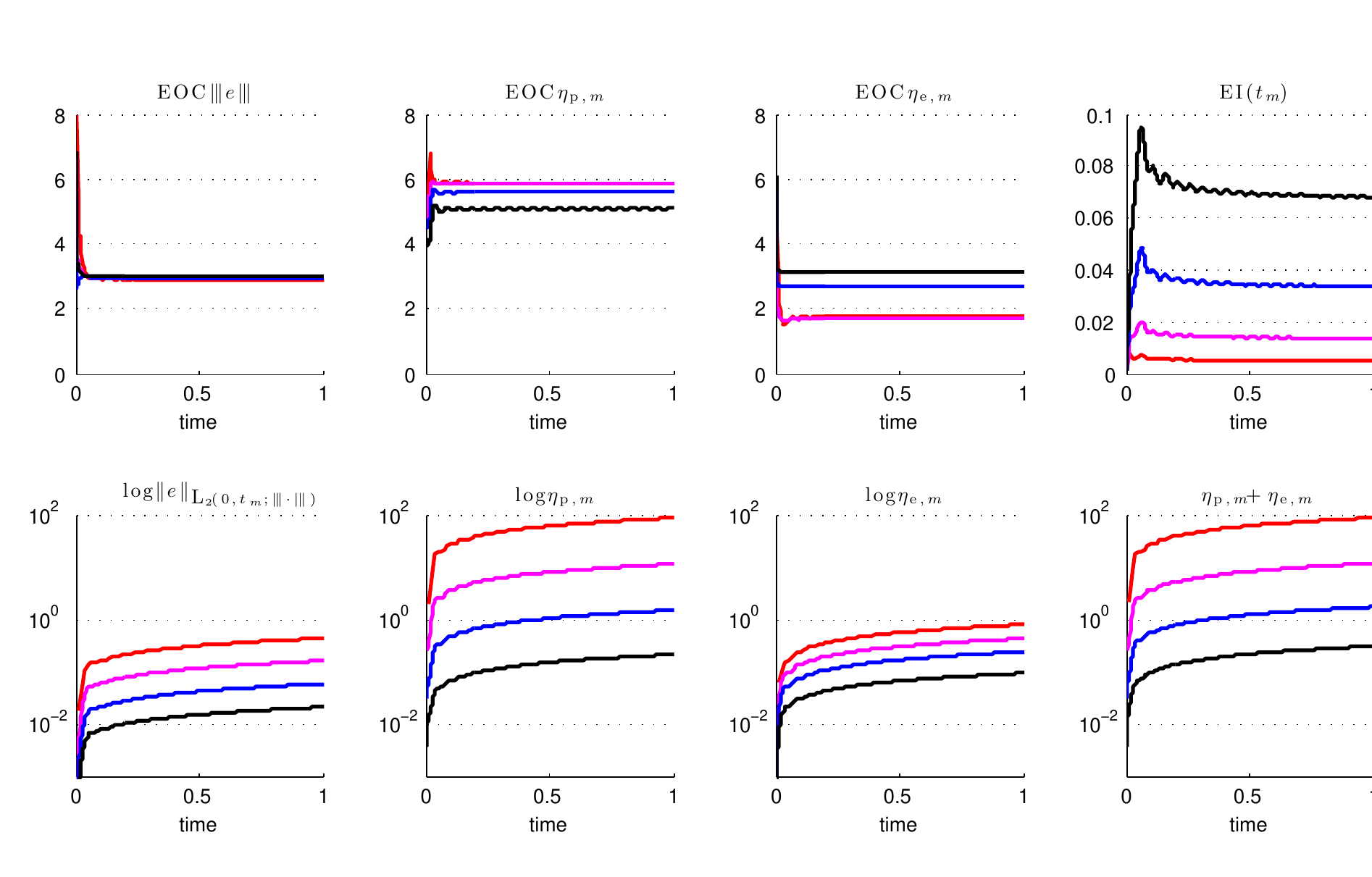}
    }
    \subfigure[{
        Mesh-size $h(i)=2^{-i/2}$, $i=2,\dotsc,4$ and timestep
        $\tau=0.1\,h^4$. On top we plot the EOC of the single
        cumulative indicators $\parest m$ and $\ellest m$. Both
        indicators have the same asymptotic $EOC\approx1$ as has the
        error. The effectivity index tends towards asymptotic value
        $200$
    }]{
      \includegraphics[trim =25pt 250pt 25pt 245pt, clip=true,width=\figwidth]{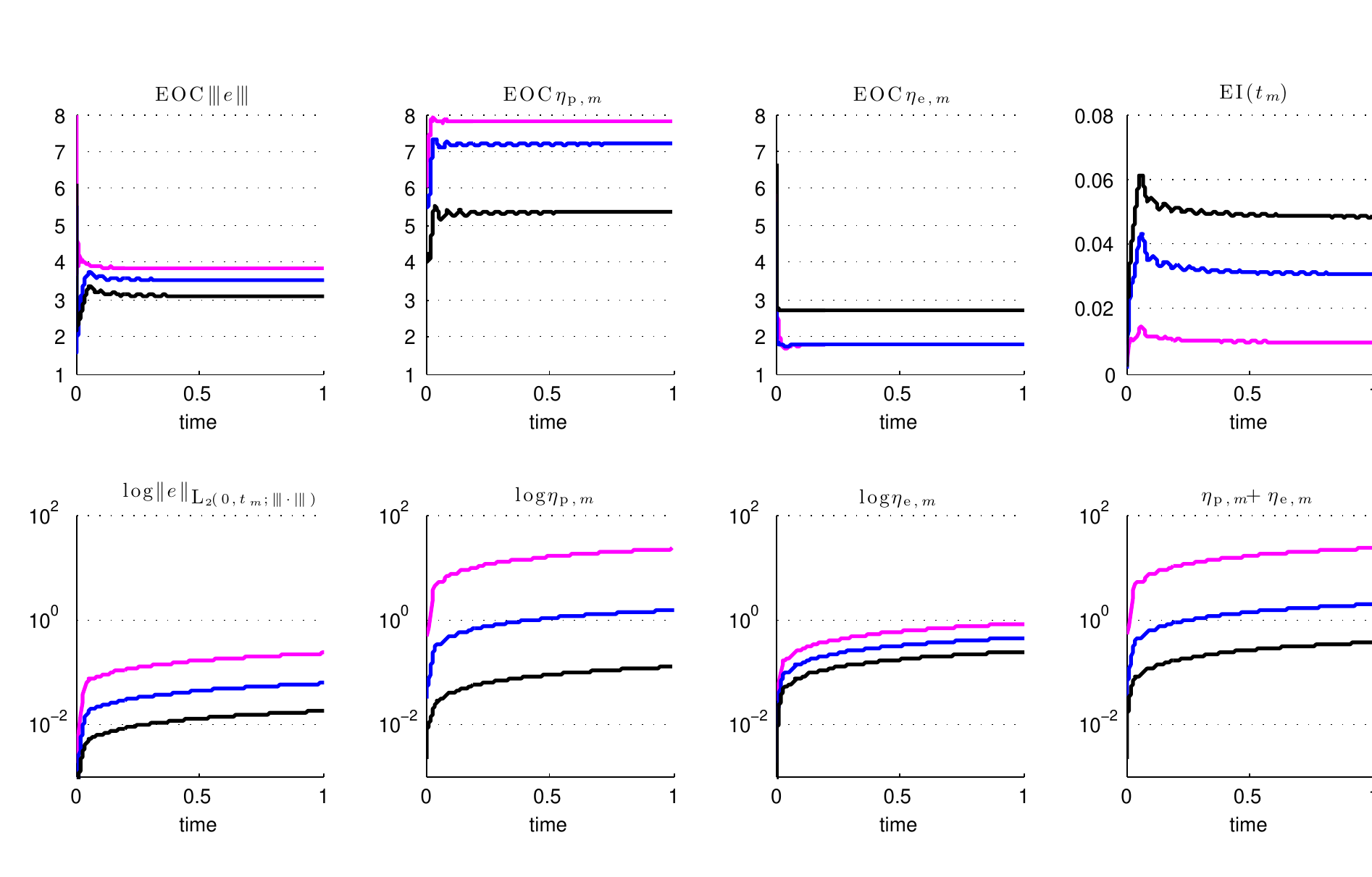}
       
    }
  \end{center}
\end{figure}

\def\cprime{$'$} \def\cprime{$'$} \def\cprime{$'$} \def\cprime{$'$}

\end{document}